\documentclass[11pt,oneside,reqno]{amsart}
\usepackage{amsmath,amssymb,amsthm,enumerate,url,graphicx}
\usepackage{amsaddr}

\newcommand{\prn}[1]{\left(#1\right)}
\newcommand{\brk}[1]{\left[#1\right]}
\newcommand{\BRK}[1]{\left\{#1\right\}}
\newcommand{\abs}[1]{\left|#1\right|}
\newcommand{\ud}[1]{\, \mathrm{d}#1}
\newcommand{\pd}[2]{\frac{\partial#1}{\partial#2}}
\newcommand{\sign}{{\rm sign}}
\newcommand{\goto}{{\rightarrow}}

\theoremstyle{plain}
\newtheorem{lem}{Lemma} %[section]
\newtheorem{thm}[lem]{Theorem} %[section]
\newtheorem{cor}[lem]{Corollary} %[section]
\theoremstyle{definition}
\newtheorem{rem}[lem]{Remark} %[section]

\begin{document}
%=============================================================================================
\title[Exactly conservative particle method for conservation laws]
{An exactly conservative particle method for one dimensional scalar conservation laws}

\author{Yossi Farjoun \and Benjamin Seibold}
\address{Department of Mathematics \\ Massachusetts Institute of Technology \\
77 Massachusetts Avenue \\ Cambridge MA 02139, USA}
\email[Y.~Farjoun]{yfarjoun@math.mit.edu}
\email[B.~Seibold]{seibold@math.mit.edu}


\thanks{The authors would like to thank R.~LeVeque for helpful comments and
suggestions. The support by the National Science Foundation is acknowledged.
Y.~Farjoun was partially supported by grant DMS--0703937.
B.~Seibold was partially supported by grant DMS--0813648.}

%\date{}
%=============================================================================================
\begin{abstract}
A particle scheme for scalar conservation laws in one space dimension is presented.
Particles representing the solution are moved according to their characteristic velocities.
Particle interaction is resolved locally, satisfying exact conservation of area.
Shocks stay sharp and propagate at correct speeds, while rarefaction waves are created
where appropriate. The method is variation diminishing, entropy decreasing, exactly
conservative, and has no numerical dissipation away from shocks. Solutions, including the
location of shocks, are approximated with second order accuracy. Source terms can be
included. The method is compared to CLAWPACK in various examples, and found to yield a
comparable or better accuracy for similar resolutions.
\end{abstract}
%=============================================================================================
\maketitle

%=============================================================================================
\section{Introduction}
%=============================================================================================
Conservation laws are important models for the evolution of continuum quantities,
describing shocks and rarefaction behavior. Fundamental mathematical properties are global
and local conservation, the presence of similarity solutions, and the method of
characteristics. Successful numerical methods employ these properties to their advantage:
Finite difference methods yield correct shock speeds if applied in conservation form.
Finite volume methods are fundamentally based on conservation properties. Godunov
schemes \cite{Godunov1959}, front tracking methods \cite{HoldenHoldenHeghKrohn1988}, and
many related approaches, approximate the global solution by local similarity solutions.
The method of characteristics is used in the CIR method \cite{CourantIsaacsonRees1952} in
combination with an interpolation scheme. Although for scalar equations it provides a
direct formula for the solution (where it is smooth), it is less popular, since it does
not possess conservation properties. Consequently, basic CIR schemes do not yield correct
shock speeds.

Many commonly used numerical methods operate on a fixed Eulerian grid. Advantages are
simple data structures and an easy generalization to higher space dimensions. Eulerian
schemes can be constructed by tracking the ``correct'' approximate solution for a short
time step, either by solving local Riemann problems (Godunov \cite{Godunov1959}) or by
tracing characteristics (CIR), followed by an interpolation step, at which the solution
is remapped onto the fixed grid. This ``remeshing'' step generally yields numerical
dissipation and dispersion. Since the shortest interaction time between shocks or
characteristics determines the global time step, remeshing is performed unnecessarily
in many places. In practice, Eulerian methods require sophisticated schemes to obtain
solutions with sharp features, but without creating oscillation. Finite volume methods
are equipped with limiters \cite{VanLeer1974}, while finite difference methods use
nonlinear approximations, such as ENO \cite{HartenEngquistOsherChakravarthy1987}
or WENO \cite{LiuOsherChan1994}.

An alternative approach is to abandon the Eulerian property, and thus avoid remeshing.
Godunov methods become front tracking methods, at least in one space dimension. While in
the former the interaction of shocks is avoided by remeshing, in the latter it is resolved
after approximating the flux function by a piecewise linear function. By construction,
front tracking is successful when representing shocks, but cumbersome when approximating
smooth parts of the solution. Similarly, CIR methods become Lagrangian particle methods.
Particles carry function values and move with their characteristic velocities.
As motivated in \cite{FarjounSeibold2009_1}, this provides a simple and accurate solution
method for conservation laws, without ever approximating derivatives. 
However, particle management is required, for two reasons: First, neighboring particles
may depart from each other, resulting in poorly resolved regions. This is prevented by
inserting particles into gaps. Second, particles may collide. If left unchecked, such a
shock event leads to a ``breaking wave'' solution. This is prevented by merging
particles upon collision.

Lagrangian particle methods have been successfully applied in the simulation of fluid
flows. Examples are vortex methods \cite{Chorin1973}, smoothed particle hydrodynamics
(SPH) \cite{Lucy1977,GingoldMonaghan1977,Monaghan2005},
or generalized SPH methods \cite{Dilts1999}.
The solution is approximated on a cloud of points which move with the flow, thus the
governing equations are discretized in their more natural Lagrangian frame of reference.
In specific applications, more accurate solutions may be obtained than with fixed grid
approaches. In addition, with particles local adaptivity is a straightforward extension.

The particle method presented here combines the method of characteristics
(where the solution is smooth) and particle merges (at shocks). The evolution of
area between neighboring particles is derived from local similarity solutions.
The method is designed to conserve area exactly.

%---------------------------------------------------------------------------------------------
\subsection{Formulation of the Particle Method}
\label{subsec:particle_method}
%---------------------------------------------------------------------------------------------
The simplest form of a one dimensional scalar conservation law is
\begin{equation}
u_t+f(u)_x = 0, \quad u(x,0) = u_0(x)
\label{eq:conservation_law_space_indep}
\end{equation}
with $f'$ continuous. The characteristic equations \cite{Evans1998}
\begin{equation}
\begin{cases}
\dot x = f'(u) \\
\dot u = 0
\end{cases}
\label{eq:characteristic_equations_space_indep}
\end{equation}
yield the solution (while it is smooth) forward in time:
At each point $(x_0,u_0(x_0))$ a characteristic curve $x(t) = x_0+f'(u_0(x_0))t$ starts,
carrying the function value $u(x(t),t) = u_0(x_0)$. While the particle method is
presented here for the simple case \eqref{eq:conservation_law_space_indep}, the method of
characteristics applies in more general cases, such as space-dependent flux functions and
source terms (see Sect.~\ref{sec:sources}). When characteristic curves collide, a shock
arises. It moves at a speed so that area (under the function $u(\cdot,t)$) evolves
correctly with respect to \eqref{eq:conservation_law_space_indep}. The Rankine-Hugoniot
condition \cite{Evans1998} follows from this principle. If the flux function $f$ is convex
or concave between the left and right state of a discontinuity, then the solution forms
either a shock or a rarefaction wave, i.e.~a continuous function connecting the two
states. Otherwise, combinations of shocks and rarefactions can result. These physical
solutions are defined by a weak formulation of \eqref{eq:conservation_law_space_indep}
accompanied by an entropy condition \cite{Evans1998}.

The first step in a particle method is to approximate the initial function $u_0$ by a
finite number of points $x_1\le\dots\le x_m$ with function values $u_1,\dots,u_m$.
In Sect.~\ref{sec:sampling_initial_data}, we present strategies on how to sample the
initial function ``well''.
The evolution of the solution is found by moving each particle $x_i$ with speed $f'(u_i)$.
This is possible as long as there are no ``collisions'' between particles. Two neighboring
particles $x_i(t)$ and $x_{i+1}(t)$ collide at time \mbox{$t+\varDelta t_i$}, where
\begin{equation}
\varDelta t_i = -\frac{x_{i+1}-x_i}{f'(u_{i+1})-f'(u_i)}\;.
\label{eq:intersection_time}
\end{equation}
A positive $\varDelta t_i$ indicates that the two particles at $x_i$ and $x_{i+1}$ will
eventually collide. Thus, $t+\varDelta t_{\text s}$ is the time of the next particle
collision, where
\begin{equation}
\varDelta t_{\text s} = \min\BRK{\BRK{\varDelta t_i | \varDelta t_i\ge 0}\cup\infty}\;.
\label{eq:time_step}
\end{equation}
For any time increment $\varDelta t\le \varDelta t_{\text s}$ the particles can be moved
directly to their new positions $x_i+f'(u_i)\varDelta t$. Thus, we can step forward 
in time an amount $\varDelta t_{\text s}$. Then, at least one particle will share its
position with another. To proceed further, we merge each such pair of particles.
If the collision time $\varDelta t_i$ is negative, the particles depart from each other.
Although at each of the particles the correct function value is preserved, after some
time their distance may be unsatisfyingly large, as the amount of error introduced during
a merge grows with the size of the neighboring gaps. To avoid this, we insert new
particles into large gaps (see Sect.~\ref{subsec:particle_management}) \emph{before}
merging particles.

In this paper, we present a method of merging and inserting particles in such a way that
shocks move at correct speeds, and rarefactions have the correct shape. The strategy is
based on mimicking the evolution of area for a conservation law, as is derived in
Sect.~\ref{sec:conservation_law_area}. The definition of an area function gives rise
to a natural interpolation between neighboring Lagrangian particles. As presented in
Sect.~\ref{sec:interp_particle_management}, particle management can then be
done to conserve area exactly. The resulting particle method is shown to be TVD.
Since the characteristic equation is solved exactly, and particle management is purely
local, the method yields no numerical dissipation (where solutions are smooth) and
correct shock speeds (where they are not). Specific strategies for sampling the initial
data are discussed in Sect.~\ref{sec:sampling_initial_data}.

In the remaining sections, the method is analyzed and generalized.
In Sect.~\ref{sec:entropy}, we prove that the numerical solutions satisfy the
Kru\v zkov entropy condition, thus showing that the method yields entropy solutions for
convex entropy functions. In Sect.~\ref{sec:inflection_points} we
present how non-convex flux functions can be treated. Strategies to include sources are
presented in Sect.~\ref{sec:sources}. In Sect.~\ref{sec:numerical_results}, we apply the
method to examples and compare it to traditional finite volume methods using
CLAWPACK \cite{Clawpack}. Conclusions are drawn in Sect.~\ref{sec:outlook}, as well as
possible applications and extensions of the method outlined.

%=============================================================================================
\section{Evolution of Area for Scalar Conservation Laws}
\label{sec:conservation_law_area}
%=============================================================================================
Consider a one dimensional scalar conservation law
\begin{equation}
u_t+f(x,u)_x = 0, \quad u(x,0) = u_0(x)\;.
\label{eq:conservation_law}
\end{equation}
Its characteristic equations \cite{Evans1998}
\begin{equation}
\begin{cases}
\dot x = f_u(x,u) \\
\dot u = -f_x(x,u)
\end{cases}
\label{eq:characteristic_equations}
\end{equation}
yield the movement and change of function value of a particle. Let $u(x,t)$ be a solution
of \eqref{eq:conservation_law}. The change of area between two \emph{fixed} points $x_1$
and $x_2$ is solely given by the flux function $f$ as
\begin{equation}
\frac{d}{dt}\int_{x_1}^{x_2}u(x,t)\ud{x} = f(x_1,u(x_1,t))-f(x_2,u(x_2,t))
= -\brk{f}_{x_1}^{x_2}\;.
\label{eq:cons_law_area_Eulerian}
\end{equation}
In contrast, the change of area between two Lagrangian particles $(x_1(t),u_1(t))$
and $(x_2(t),u_2(t))$, i.e.~points that \emph{move} according to
\eqref{eq:characteristic_equations}, is given by
\begin{align}
\frac{d}{dt}\int_{x_1(t)}^{x_2(t)}u(x,t)\ud{x}
&= \prn{f_u(x_2,u_2)u_2-f(x_2,u_2)}-\prn{f_u(x_1,u_1)u_1-f(x_1,u_1)} \nonumber \\
&= F(x_2,u_2)-F(x_1,u_1) = \brk{F}_{(x_1,u_1)}^{(x_2,u_2)}\;,
\label{eq:cons_law_area_Lagrangian}
\end{align}
where $F=f_uu-f$ is the Legendre transform of $f$. That is, $f$ is a Hamiltonian of the
dynamics \eqref{eq:characteristic_equations}, and $F$ is a Lagrangian.
Equation \eqref{eq:cons_law_area_Eulerian} (respectively \eqref{eq:cons_law_area_Lagrangian})
yields the change of area between two Eulerian (Lagrangian) points, only by knowing the
flux $f$ (the Lagrangian $F$) at the two points. Hence, in the same fashion
as \eqref{eq:cons_law_area_Eulerian} can be used to construct a conservative fixed grid
method, we use \eqref{eq:cons_law_area_Lagrangian} to construct a
conservative \emph{particle} method.

Consider an \emph{area value} $A_i(t)$ associated with each particle, such that
$\brk{A}_{x_i}^{x_{i+1}} = A_{i+1}-A_i$ is the area between $x_i$ and $x_{i+1}$.
Assume the values $A_i$ are known at $t=0$. Then we can find the areas at any time
by solving the system arising from equations \eqref{eq:characteristic_equations}
and \eqref{eq:cons_law_area_Lagrangian}
\begin{equation}
\begin{cases}
\dot x_i = f_u(x_i,u_i) \\
\dot u_i = -f_x(x_i,u_i) \\
\dot A_i = F(x_i,u_i)\;.
\end{cases}
\end{equation}
\begin{rem}
\label{rem:F0}
While $\dot f = 0$ (since $f$ is a Hamiltonian of the dynamics), in general
$\dot F\neq 0$. However, if the flux function satisfies
\begin{equation}
f_{xu}f_uu-f_{uu}f_xu-f_xf_u = 0\;,
\label{eq:flux_F0}
\end{equation}
then $\dot F = 0$ (by the chain rule). Property \eqref{eq:flux_F0} is
satisfied for instance if $f = f(u)$ or $f(x,u) = \varphi(x)u^k$. If $\dot F=0$, the
evolution of area is particularly simple, namely $A_i$ changes at a \emph{constant}
rate $F_i$.
\end{rem}

%---------------------------------------------------------------------------------------------
\subsection{Space-independent Flux}
\label{subsec:conservation_law_area_similarity}
%---------------------------------------------------------------------------------------------
Henceforth we only consider flux functions that are independent of the spatial
variable, $f = f(u)$. Thus, by Rem.~\ref{rem:F0}, the area between two Lagrangian
points changes linearly, as does the distance between them
\begin{align}
\frac{d}{dt}\int_{x_1(t)}^{x_2(t)} u(x,t)\ud{x} &= \brk{F(u)}_{u_1}^{u_2}\;,
\label{eq:area_deriv} \\
\frac{d}{dt}(x_2(t)-x_1(t)) &= \dot{x}_2(t)-\dot{x}_1(t) = f'(u_2)-f'(u_1)
= \brk{f'(u)}_{u_1}^{u_2}\;.
\label{eq:distance_deriv}
\end{align}
If the two points $x_1$ and $x_2$ move at different speeds, then there is a time $t_0$
(which may be larger or smaller than $t$) at which they have the same position. This
assumes that they remain characteristic points between $t$ and $t_0$, i.e.~they do not
interact with shocks. At time $t_0$, the distance and the area between the two points
vanish. From \eqref{eq:area_deriv} and \eqref{eq:distance_deriv} we have that
\begin{align*}
\int_{x_1(t)}^{x_2(t)} u(x,t)\ud{x} &= (t-t_0) \cdot \brk{F(u)}_{u_1}^{u_2} \;, \\
x_2(t)-x_1(t) &= (t-t_0) \cdot \brk{f'(u)}_{u_1}^{u_2}\;.
\end{align*}
In short, the area between two Lagrangian points can be written as
\begin{equation}
\int_{x_1(t)}^{x_2(t)} u(x,t)\ud{x} = (x_2(t)-x_1(t))\,a_f(u_1,u_2) \;,
\label{eq:area}
\end{equation}
where $a_f(u_1,u_2)$ is the nonlinear average function
\begin{equation}
a_f(u_1,u_2)
= \frac{\brk{f'(u)u-f(u)}_{u_1}^{u_2}}{\brk{f'(u)}_{u_1}^{u_2}}
= \frac{\int_{u_1}^{u_2}f''(u)\,u\ud{u}}{\int_{u_1}^{u_2}f''(u)\ud{u}}\;.
\label{eq:nonlinear_average}
\end{equation}
If there is only one flux function, we drop the subscript, and simply write $a(u_1,u_2)$.
The integral form shows that $a$ is indeed an average of $u$, weighted by $f''$.
The evolution of area \eqref{eq:area} is independent of the specific solution, since by
assumption we have excluded all solutions for which a shock would interact with either
characteristic point. The following lemma describes some properties of the nonlinear
average $a(\cdot,\cdot)$.
\begin{lem}
\label{thm:average_properties}
Let $f$ be strictly convex in $[u_{{}_L},u_{{}_U}]$, that is,
$f''>0$ in $(u_{{}_L},u_{{}_U})$.
Then for all $u_1, u_2\in [u_{{}_L},u_{{}_U}]$, the
average \eqref{eq:nonlinear_average}
is\dots
\begin{enumerate}
\item the same for $f$ and $-f$;
\item symmetric, $a(u_1,u_2) = a(u_2,u_1)$;
\item an average, i.e.~$a(u_1,u_2)\in(u_1,u_2)$, for $u_1\neq u_2$;
\item strictly increasing in both $u_1$ and $u_2$; and
\item continuous at $u_1=u_2$, with $a(u,u)=u$.
\end{enumerate}
\end{lem}
\noindent
Due to the first two properties, we can assume WLOG that $f''>0$ and $u_1\le u_2$
whenever convenient.
\begin{proof}
We prove the claims in turn.
\begin{enumerate}
\item[(1,2)] Multiplying both numerator and denominator by $-1$ yields the proof:
\begin{align*}
a_f(u_1,u_2) &= \frac{-\int_{u_1}^{u_2}f''(u)\,u
\ud{u}}{-\int_{u_1}^{u_2}f''(u) \ud{u}}
= \frac{\int_{u_1}^{u_2}-f''(u)\,u \ud{u}}{\int_{u_1}^{u_2}-f''(u)
\ud{u}}= a_{-f}(u_1,u_2)\\
&=\frac{\int_{u_2}^{u_1}f''(u)\,u \ud{u}}{\int_{u_2}^{u_1}f''(u) \ud{u}}= a_f(u_2,u_1)\;.
\end{align*}
\addtocounter{enumi}{2}
\item\label{prop:avg}
We bound $a$ from above:
\begin{equation*}
a(u_1,u_2) =  \frac{\int_{u_1}^{u_2}f''(u)\,u \ud{u}}{\int_{u_1}^{u_2}f''(u) \ud{u}}
< \frac{u_2\int_{u_2}^{u_1}f''(u) \ud{u}}{\int_{u_2}^{u_1}f''(u) \ud{u}}= u_2\;.
\end{equation*}
A similar argument bounds $a$ from below.
\item
We show that $a(u_1,u_2)$ is strictly increasing in the second argument.
Let $u_1<u_2<u_3$, $u_i\in [u_{{}_L}, u_{{}_U}]$. Then
\begin{align*}
a(u_1, u_3)
&=\frac{\int_{u_1}^{u_2}f''(u)u\ud{u}+\int_{u_2}^{u_3}f''(u)u\ud{u}}
{\int_{u_1}^{u_3}f''(u)\ud{u}} \\
&=\frac{a(u_1,u_2)\int_{u_1}^{u_2}f''(u)\ud{u}+a(u_2,u_3)\int_{u_2}^{u_3}f''(u)\ud{u}}
{\int_{u_1}^{u_3}f''(u)\ud{u}}\;.
\end{align*}
Due to property (\ref{prop:avg}) we have that $a(u_1,u_2)<u_2<a(u_2,u_3)$. Thus
\begin{align*}
a(u_1,u_3)&>\frac{a(u_1,u_2)\int_{u_1}^{u_2}f''(u)\ud{u}
+a(u_1,u_2)\int_{u_2}^{u_3}f''(u)\ud{u}}
{\int_{u_1}^{u_3}f''(u)\ud{u}}
= a(u_1, u_2)\;.
\end{align*}
A similar argument shows the result for the first argument.
\item 
Since $u_1<a(u_1,u_2)<u_2$ for all $u_1\ne u_2$,
we have (by the Sandwich Theorem) that 
\begin{equation*}
u = \lim_{u_1\goto u}u_1\le \lim_{u_1, u_2\goto u} a(u_1,u_2)\le\lim_{u_2\goto u}u_2 = u\;.
\end{equation*}
Therefore, $\lim\limits_{u_1, u_2\goto u} a(u_1,u_2)=u$.
\qedhere
\end{enumerate}
\end{proof}

%=============================================================================================
\section{Interpolation and Particle Management}
\label{sec:interp_particle_management}
%=============================================================================================
The time evolution of equation \eqref{eq:conservation_law_space_indep} is described by the
characteristic movement of the particles \eqref{eq:characteristic_equations}. Particle
management is an ``instantaneous'' operation (i.e.~happening at constant time)
that allows the method to continue stepping
forward in time. It is designed to conserve area: The function value of an inserted or
merged particle is chosen such that area is unchanged by the operation. A simple condition
guarantees that the entropy does not increase. In addition, we define an interpolating
function between two neighboring particles, so that the change of area under the
interpolating curve satisfies relation \eqref{eq:area_deriv}. This interpolation is shown
to be an analytical solution of the conservation law.

%---------------------------------------------------------------------------------------------
\subsection{Conservative Particle Management}
\label{subsec:particle_management}
%---------------------------------------------------------------------------------------------
Consider four neighboring particles located at
\mbox{$x_1<x_2\le x_3<x_4$}\footnote{If more than two particles are at one position ($x$),
all but the one with the smallest value ($u$) and the one with the largest value ($u$)
are removed immediately.}
with associated function values $u_1$, $u_2$, $u_3$, $u_4$. Assume that the flux $f$ is
strictly convex or concave on the range of function values $[\min_i(u_i),\max_i(u_i)]$.
If $u_2\neq u_3$, the particles' velocities must differ $f'(u_2)\neq f'(u_3)$, which
gives rise to two possible cases that require particle management:
\begin{itemize}
\item \textbf{Inserting:}
The two particles deviate, i.e.~$f'(u_2)<f'(u_3)$. If $x_3-x_2\ge d_\text{max}$ for
some predefined maximum distance $d_\text{max}$, we insert a new particle
$(x_{23},u_{23})$ with $x_2<x_{23}<x_3$, such that the area is preserved:
\begin{equation}
(x_{23}-x_2)\,a(u_2,u_{23})+(x_3-x_{23})\,a(u_{23},u_3) = (x_3-x_2)\,a(u_2,u_3) \;.
\label{eq:area_cond_insert}
\end{equation}
One can, for example, set $x_{23}=\frac{x_2+x_3}{2}$ and find $u_{23}$
by \eqref{eq:area_cond_insert}, or set $u_{23}=\frac{u_2+u_3}{2}$ and find $x_{23}$
by \eqref{eq:area_cond_insert}.
\item \textbf{Merging:}
The two particles collide, i.e.~$f'(u_2)>f'(u_3)$. If $x_3-x_2\le d_\text{min}$ for
some predefined minimum distance ($d_\text{min}=0$ is possible), we replace
both with a new particle $(x_{23},u_{23})$ with $x_2\le x_{23}\le x_3$, such that the
area is preserved:
\begin{align}
(x_{23}-x_1)\,a(u_1,u_{23})+(x_4-x_{23})\,a(u_{23},u_4)& \label{eq:area_cond_merge} \\
= (x_2\!-\!x_1)\,a(u_1,u_2)+(x_3\!-\!x_2)\,a(u_2,u_3)&+(x_4\!-\!x_3)\,a(u_3,u_4)\;.
\nonumber
\end{align}
We choose $x_{23}=\frac{x_2+x_3}{2}$, and then find $u_{23}$ such that
\eqref{eq:area_cond_merge} is satisfied.
Figure~\ref{fig:merging} illustrates the merging step.
\end{itemize}
Observe that inserting and merging are similar in nature.
Conditions \eqref{eq:area_cond_insert} and \eqref{eq:area_cond_merge} for $u_{23}$ are
nonlinear (unless $f$ is quadratic, see Rem.~\ref{rem:quadratic_flux}).
For most cases $u_{23}=\frac{u_2+u_3}{2}$ is a good initial guess, and the correct value
can be obtained (up to the desired precision) by a few Newton iteration steps
(or bisection, if the Newton iteration fails to converge). The next few claims attest
that there is a unique value $u_{23}$ that satisfies \eqref{eq:area_cond_insert}
and \eqref{eq:area_cond_merge}, respectively.
\begin{lem}
The function value $u_{23}$ for the particle at $x_{23}$ for
equations \eqref{eq:area_cond_insert} and \eqref{eq:area_cond_merge}
is unique.
\end{lem}
\begin{proof}
{}From Lemma~\ref{thm:average_properties} we have that both $a(u_1,\cdot)$ and
$a(\cdot,u_4)$ are strictly increasing. Thus, the LHS of both \eqref{eq:area_cond_insert}
and \eqref{eq:area_cond_merge} are strictly increasing (in $u_{23}$), and cannot attain
the same value for different values of $u_{23}$.
\end{proof}
\begin{lem}
\label{lem:existence_u23_insert}
There exists $u_{23}\in [u_2,u_3]$ which
satisfies \eqref{eq:area_cond_insert}. 
\end{lem} 
\begin{proof}
WLOG we assume $u_2\le u_3$. We define 
\begin{align*}
A\phantom{(u)}&=(x_3-x_2)a(u_2,u_3)\\
B(u)&=(x_{23}-x_2)a(u_2,u)+(x_3-x_{23})a(u,u_3)\;.
\end{align*}
So equation \eqref{eq:area_cond_insert} can be recast as $B(u_{23})=A$.
The monotonicity of $a(\cdot,\cdot)$ implies that $B(u_2)<A<B(u_3)$.
Since $a$ is continuous, so is $B$, and the result follows from the Intermediate Value
Theorem. The proof for $u_2>u_3$ is trivially similar.
\end{proof}
\begin{lem}
\label{lem:existence_u23}
If $x_2 = x_3 = x_{23}$, there exists $u_{23}\in [u_2,u_3]$ which
satisfies \eqref{eq:area_cond_merge}. 
\end{lem} 
\begin{proof}
The proof is identical to the proof of Lemma~\ref{lem:existence_u23_insert} with
the following definition of $A$ and $B(u)$:
\begin{align*}
A\phantom{(u)}&=(x_2-x_1)\,a(u_1,u_2)+(x_4-x_2)\,a(u_3,u_4)\\
B(u) &= (x_2-x_1)\,a(u_1,u\phantom{{}_2})+(x_4-x_2)\,a(u\phantom{{}_3},u_4)\;.
%\qedhere
\end{align*}
\end{proof}
\begin{cor}
If particles are merged and inserted according to equations 
(\ref{eq:area_cond_insert}, \ref{eq:area_cond_merge}), then the total variation of the
solution is either the same as before the operation, or smaller.
\end{cor}
Merging points only when $x_2=x_3$ can be overly restrictive.
The following theorem grants a little more freedom.
\begin{thm}
\label{thm:merging_TVD}
Consider four consecutive particles $(x_i,u_i)\;i=1,\ldots,4$. If
\begin{equation}
\frac{\abs{u_3-u_2}}{x_3-x_2}
\ge 4\prn{\frac{\max\abs{f''}}{\min\abs{f''}}}^6
\frac{\max_i u_i-\min_i u_i}{\min\prn{x_4-x_3,x_2-x_1}}\;,
\label{eq:in_box:cond}
\end{equation}
then merging particles 2 and 3 with $x_{23}=\frac{x_2+x_3}{2}$ yields
$u_{23}\in[u_2,u_3]$.
\end{thm}
The $\min$ and $\max$ of $f''$ are taken over the maximum range of $u_1,\dots,u_4$. 
Condition \eqref{eq:in_box:cond} is trivially satisfied if $x_2=x_3$.

The idea of the proof is to consider merging in two steps. First, we find a value
$\tilde u$ such that setting $u_2=u_3=\tilde u$ (while leaving $x_2$ and $x_3$ unchanged)
preserves the area. Next, we merge the two particles to one with value $u_{23}$
located at $x_{23}$. To prove the theorem we use two lemmas:
Lemma~\ref{lem:merge:lower_bound} bounds $\tilde u$ \emph{away} from $u_2$ and $u_3$
(but inside $[u_2,u_3]$). 
Lemma~\ref{lem:merge:upper_bound} bounds $\abs{u_{23}-\tilde u}$ from above.
We define three ``area functions'':
\begin{align*}
A\phantom{(u)}&= (x_2-x_1)a(u_1,u_2) + (x_3-x_2)a(u_2,u_3)+(x_4-x_3)a(u_3,u_4)\;, \\
B(u)&=(x_2-x_1)a(u_1,u) + (x_3-x_2)a(u,u) + (x_4-x_3)a(u,u_4)\;, \\
C(u)&=(x_2-x_1)a(u_1,u) + (x_3-x_2)\tfrac12\brk{a(u_1,u)+a(u,u_4)}+(x_4-x_3)a(u,u_4)\;.
\end{align*}
Here $A$ is the area before the merge that needs to be preserved, $B(u)$ is the area when
the particles $2, 3$ have the value $u$, and $C(u)$ is the area when particles $2, 3$
have been merged to a single particle at $\frac{x_2+x_3}{2}$ with value $u$.
\begin{lem}\label{lem:merge:lower_bound}
The value $\tilde u$ for which $B(\tilde u) = A$ satisfies $\tilde u\in[u_2, u_3]$ and
\begin{equation*}
\abs{\tilde u - u_i}\ge\frac12
\frac{\min(x_3-x_1,x_4-x_2)\abs{u_3-u_2}}{x_4-x_1}\abs{\frac{\min f''(u)}{\max f''(u)}}^4
\text{ for } i=2, 3\;.
\end{equation*}
\end{lem}
\begin{lem}\label{lem:merge:upper_bound}
The value $u_{23}$ for which $C(u_{23})=A$ satisfies 
\begin{equation*}
\abs{u_{23}-\tilde u}\le 2
\frac{(x_3-x_2)\brk{\max(u_1,\,u_2,\,u_3,\,u_4)-\min(u_2,u_3)}}{x_4-x_1}
\abs{\frac{\max f''(u)}{\min f''(u)}}^2\;.
\end{equation*}
\end{lem}
The proofs of the last two lemmas are tedious and uninspiring; they are relegated to the
appendix for the interested reader's perusal.
\begin{proof}[of Thm.~\ref{thm:merging_TVD}]
Starting from the hypothesis of the theorem, we find
\begin{align}
\frac{1}{2}\frac{{\min\prn{x_4-x_3,x_2-x_1}}\abs{u_3-u_2}}{x_4-x_1}
&\prn{\frac{ \min\abs{f''}}{\max\abs{f''}}}^4\notag\\
\ge 2\prn{\frac{ \max\abs{f''}}{\min\abs{f''}}}^2
&\frac{(\max_i u_i-\min_i u_i)(x_3-x_2)}{x_4-x_1}\notag\\
\ge 2\prn{\frac{ \max\abs{f''}}{\min\abs{f''}}}^2
&\frac{(\max_i u_i-\min\prn{u_2,u_3})(x_3-x_2)}{x_4-x_1}\;.\notag
\end{align}
Using Lemmas~\ref{lem:merge:lower_bound} and~\ref{lem:merge:upper_bound}, we obtain 
\begin{equation*}
\abs{\tilde u-u_i}\ge\abs{u_{23}-\tilde u},
\end{equation*} 
for $i=2,3$. 
Since we also have that $\tilde u\in[u_2,u_3]$ (from Lemma~\ref{lem:merge:upper_bound}), 
we conclude that $u_{23}\in [u_2,u_3]$.
\end{proof}
\begin{rem}
\label{rem:merging_robust}
Due to Thm.~\ref{thm:merging_TVD}, the merging step is robust with respect to
small deviations in the distance of the merged particles. This also holds for the case
$x_2>x_3$, given the distance $\abs{x_3-x_2}$ is sufficiently small.
\end{rem}
\begin{thm}
\label{thm:arbitrary_times}
The particle method can run to arbitrary times.
\end{thm}
\begin{proof}
Let $u_{{}_L}=\min_i u_i$, $u_{{}_U}=\max_i u_i$, and
$C=\max_{[u_{{}_L},u_{{}_U}]}|f''(u)|\cdot(u_{{}_U}-u_{{}_L})$.
For any two particles, one has $|f'(u_{i+1})-f'(u_i)|\le C$. Define
$\varDelta x_i=x_{i+1}-x_i$. After each particle management, the next time increment
(as defined in Sect.~\ref{subsec:particle_method}) is at least
$\varDelta t_{\text s}\ge \frac{\min_i{\varDelta x_i}}{C}$. If we do not insert
particles, then in each merge one particle is removed. Hence, a time evolution beyond
any given time is possible, since the increments $\varDelta t_{\text s}$ will increase
eventually. When a particle is inserted (whenever two points are at a distance more
than $d_{\max}$), the created distances are at least $\frac{d_{\max}}{2}$,
preserving a lower bound on the following time increment.
\end{proof}
\begin{rem}[Choice of distance parameters]
\label{rem:distance_parameters}
If possible, the minimum particle distance should be chosen $d_\text{min} = 0$.
However, a small positive value of $d_\text{min}$ also leads to a
working method. 
This generalization is of interest if the characteristic equations
have to be solved numerically, such as in Sect.~\ref{sec:sources}.

The maximum particle distance $d_\text{max}$ gives the minimal local resolution of the
method. If the initial data is sampled with a resolution of $h$, then a good choice for
the maximum distance is $h<d_\text{max}<2h$. Here, the upper bound comes from the fact
that after an insertion the local particle distance is halved. Throughout our numerical
simulations we use $d_{\text{max}}=\frac43h$ since this gives, on average, a distance
of $h$ between particles in a rarefaction. This is an estimate that results from solving
$\tfrac12\prn{d_\text{max}+\tfrac12 d_\text{max}}=h$ for $d_\text{max}$.  
\end{rem}
\begin{figure}
\hspace{-3cm}
\begin{minipage}[t]{.65\textwidth}
\centering
\includegraphics[width=0.75\textwidth]{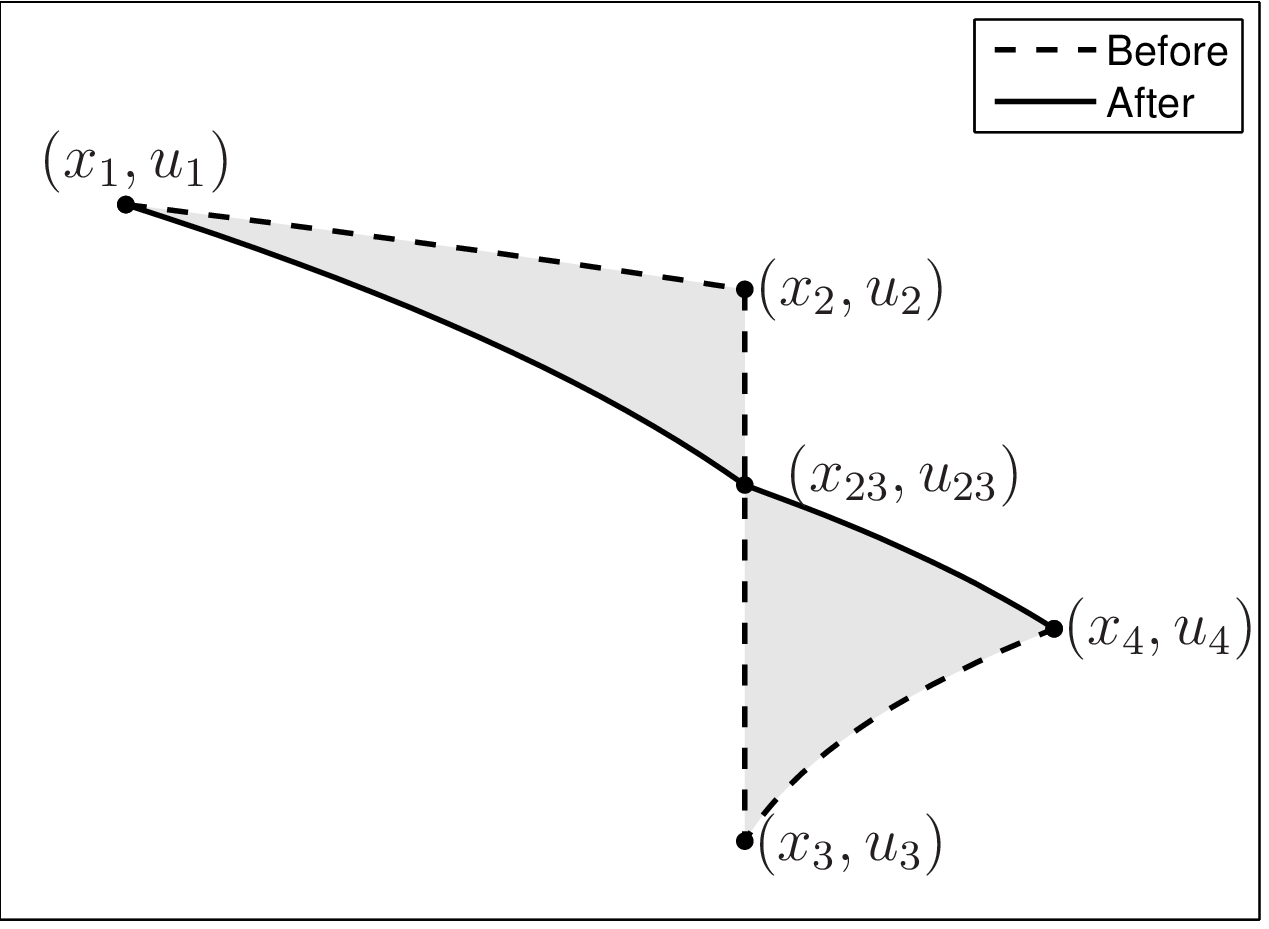}
\caption{Merging two particles yields a new particle with function value chosen
``conservatively''. This implies that the area of the two ``triangles'' is the same.}
\label{fig:merging}
\end{minipage}
\hspace{-2cm}
\begin{minipage}[t]{.65\textwidth}
\centering
\includegraphics[width=0.75\textwidth]{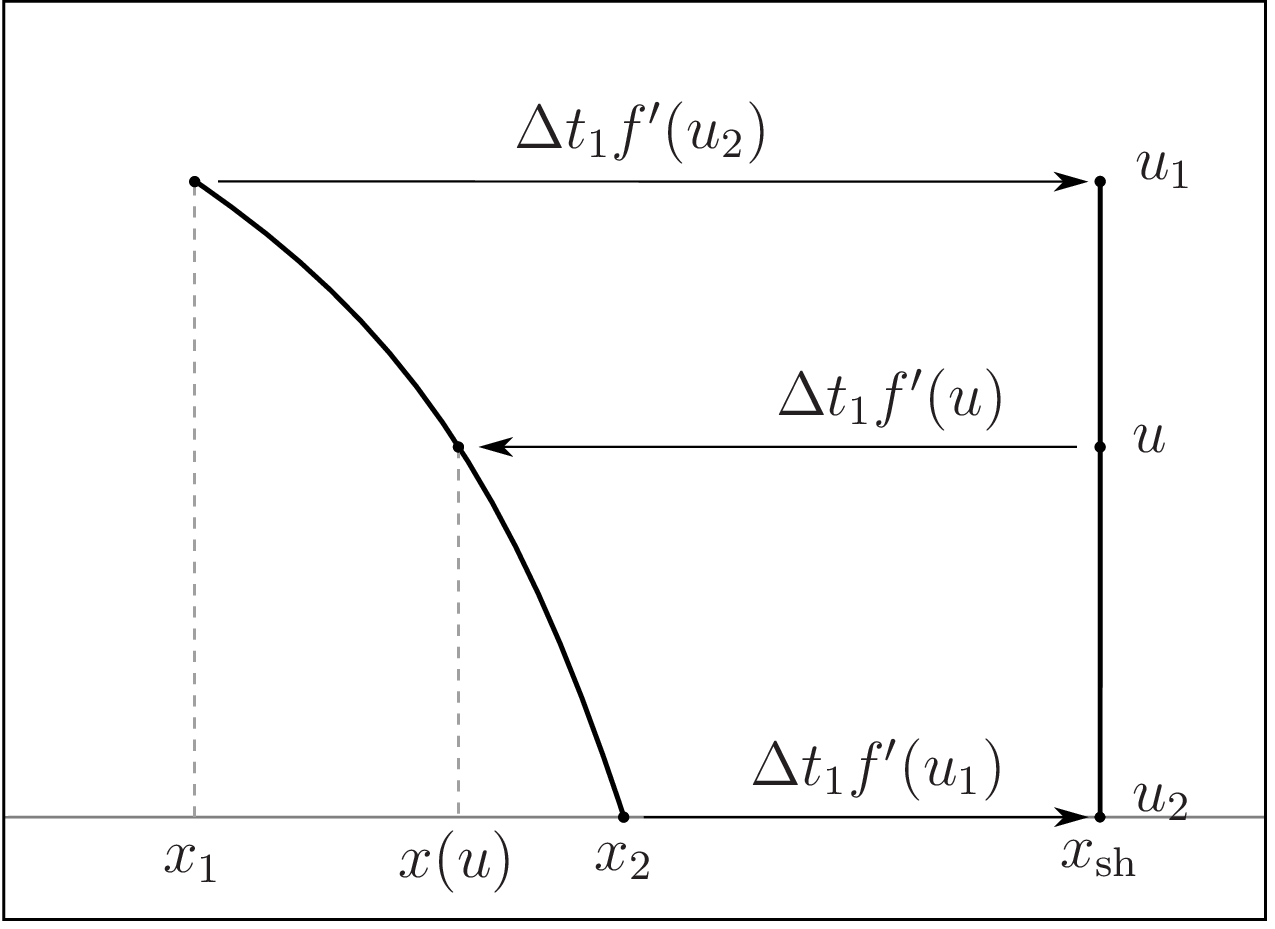}
\caption{To find the $x-$value of a particle with given $u-$value, one locates the
shock and then travels to the current time with velocity $f'(u)$ as given by
\eqref{eq:interpolation_function}.}
\label{fig:def_interp}
\end{minipage}
\hspace{-3cm}
\end{figure}

%---------------------------------------------------------------------------------------------
\subsection{Conservative Interpolation}
\label{subsec:interpolation}
%---------------------------------------------------------------------------------------------
Expression \eqref{eq:area} defines an area between any two points. We show that this
defines an interpolating function $v(x)$ between the two points. While an interpolation
is not required for the particle management itself, it is useful for plotting the
numerical solution, interpreting its properties, and including source terms.
As derived in Sect.~\ref{subsec:conservation_law_area_similarity}, the area
between two points $(x_1,u_1)$ and $(x_2,u_2)$ equals
\begin{equation*}
\int_{x_1}^{x_2}v(x)\ud{x} = (x_2-x_1)a(u_1,u_2)\;,
\end{equation*}
assuming that $f$ is strictly convex or concave in $[u_1,u_2]$.
We define the interpolation by the principle that any point $(x,v)$ on the function
$v = v(x)$ must yield the same area when the interval is split:
\begin{equation}
(x-x_1)a(u_1,v)+(x_2-x)a(v,u_2) = (x_2-x_1)a(u_1,u_2)\;.
\label{eq:area_cond_interpolant}
\end{equation}
If $u_1 = u_2$, the interpolant is a constant function. Otherwise,
\eqref{eq:area_cond_interpolant} can be rearranged to yield
\begin{equation}
\frac{x-x_1}{x_2-x_1}
= \frac{a(u_1,u_2)-a(u_2,v)}{a(u_1,v)-a(u_2,v)}
= \frac{f'(v)-f'(u_1)}{f'(u_2)-f'(u_1)}\;.
\label{eq:interpolation_function_area}
\end{equation}
The last equality in \eqref{eq:interpolation_function_area} follows from
\begin{lem}
For any $u_1,u_2,u_3$ with $u_1\neq u_2$, the nonlinear average satisfies
\begin{equation*}
\frac{a(u_1,u_2)-a(u_2,u_3)}{a(u_1,u_3)-a(u_2,u_3)}
= \frac{f'(u_3)-f'(u_1)}{f'(u_2)-f'(u_1)}\;.
\end{equation*}
\end{lem}
\begin{proof}
By definition of the average function, we have
\begin{align*}
0 &= a(u_1,u_2)\brk{f'(u)}_{u_1}^{u_2}+a(u_2,u_3)\brk{f'(u)}_{u_2}^{u_3}
+a(u_3,u_1)\brk{f'(u)}_{u_3}^{u_1} \\
&= a(u_1,u_2)\brk{f'(u)}_{u_1}^{u_2}+a(u_2,u_3)
\prn{\brk{f'(u)}_{u_1}^{u_3}-\brk{f'(u)}_{u_1}^{u_2}}
-a(u_1,u_3)\brk{f'(u)}_{u_1}^{u_3} \\
&= \prn{a(u_1,u_2)-a(u_2,u_3)}\brk{f'(u)}_{u_1}^{u_2}
-\prn{a(u_1,u_3)-a(u_2,u_3)}\brk{f'(u)}_{u_1}^{u_3}\;.
\end{align*}
Rearranging the terms proves the claim.
\end{proof}
Observe that condition \eqref{eq:area_cond_interpolant} is identical to the condition for
particle insertion \eqref{eq:area_cond_insert}, which means that any newly inserted
particle must be placed on the interpolant.
Relation \eqref{eq:interpolation_function_area} defines the interpolant as a
function $x(v)$. This is in fact an advantage, since at a discontinuity $x_1 = x_2$,
characteristics for all intermediate values $v$ are defined. Thus, rarefaction fans arise
naturally if $f'(u_1)<f'(u_2)$. If $f$ has no inflection points between $u_1$ and $u_2$,
the inverse function $v(x)$ exists. For plotting purposes we plot $x(v)$ instead of
inverting the function.

The interpolation \eqref{eq:interpolation_function_area} can also be derived as a
similarity solution of the
conservation law \eqref{eq:characteristic_equations_space_indep}, as follows.
If $u_1=u_2$, we define $v(x)=u_1$. Otherwise, one has $f'(u_1)\neq f'(u_2)$.
As derived in Sect.~\ref{sec:conservation_law_area}, the solution either came from a
discontinuity (i.e.~it is a rarefaction wave) or it will become a shock (i.e.~it is a
compression wave).
The time
$\varDelta t_1$ until this discontinuity happens is given by
\eqref{eq:intersection_time}. At time $t+\varDelta t_1$ the particles have the
same position $x_1=x_2=x_{\text{sh}}$, as shown in Fig.~\ref{fig:def_interp}.
At this time the interpolation must be a straight line connecting the two particles,
representing a discontinuity at $x_{\text{sh}}$. We require any particle of the
interpolating function $(x,v(x))$ to move with its characteristic velocity $f'(v(x))$ in
the time between $t$ and $t+\varDelta t_1$. 
This defines the interpolation uniquely as
\begin{equation}
x(v) = x_1-\varDelta t_1\prn{f'(v)-f'(u_1)}
= x_1+\frac{f'(v)-f'(u_1)}{f'(u_2)-f'(u_1)}(x_2-x_1) \;,
\label{eq:interpolation_function}
\end{equation}
which equals expression \eqref{eq:interpolation_function_area}.

Not only is this interpolation compatible with the evolution of area under a conservation
law, it also is a solution:
\begin{lem}
\label{lem:interpolation:is:solution}
Together with the characteristic motion of the nodes, interpolation
\eqref{eq:interpolation_function} is a solution of the conservation
law \eqref{eq:conservation_law}. 
\end{lem}
\begin{proof}
Using that \mbox{$\dot x_i(t)=f'(u_i)$} for~$i=1,2$ one obtains
\begin{align*}
\pd{x}{t}(v,t)&= \dot x_1
+\frac{f'(v)-f'(u_1)}{f'(u_2)-f'(u_1)}(\dot x_2-\dot x_1)\\
&=f'(u_1)+\frac{f'(v)-f'(u_1)}{f'(u_2)-f'(u_1)}(f'(u_2)-f'(u_1))
=f'(v) \;.
\end{align*}
Thus every point on the interpolation $v(x,t)$ satisfies the characteristic
equation \eqref{eq:characteristic_equations}.
\end{proof}
\begin{cor}[exact solution property]
\label{cor:exact_solution}
Consider characteristic particles with $x_1(t)<x_2(t)<\dots<x_n(t)$. 
At any time consider the function defined by the interpolation
\eqref{eq:interpolation_function}. This function is a classical
(i.e.~continuous) solution to the conservation law \eqref{eq:conservation_law}.
In particular, it satisfies the conservation properties given in
Sect.~\ref{sec:conservation_law_area}.
\end{cor}
\noindent This corollary breaks down when shocks occur.
\begin{thm}[TVD]
With the assumptions of Thm.~\ref{thm:merging_TVD}, the particle method is total
variation diminishing, thus it does not create spurious oscillations.
\end{thm}
\begin{proof}
Due to Cor.~\ref{cor:exact_solution}, the characteristic movement yields an exact
classical solution, thus the total variation is constant. Particle insertion simply
refines the interpolation, thus preserves the total variation.
Due to Thm.~\ref{thm:merging_TVD}, merging yields a new particle with a function
value $u_{23}$ between the values of the removed particles. Thus, the total variation
is the same as before the merge or smaller.
\end{proof}
\begin{rem}
The particle method approximates the solution locally by similarity solutions, very
similar to \emph{front tracking} by Holden et al.~\cite{HoldenHoldenHeghKrohn1988}.
Front tracking uses shocks (after approximating the flux function by a piecewise linear,
and the solution by a piecewise constant function). In comparison, our method uses wave
solutions, i.e.~rarefactions and compressions.
\end{rem}

%---------------------------------------------------------------------------------------------
\subsection{Computational Aspects}
%---------------------------------------------------------------------------------------------
\begin{rem}[Shock location]
\label{rem:shock_location}
The particle method does not track shocks. Still, shocks can be located. Whenever
particles are merged, the new particle can be marked as a \emph{shock particle}. Thus,
any shock stretches over three particles $(x_1,u_1),(x_2,u_2),(x_3,u_3)$, with the shock
particle in the middle. Before plotting or interpreting the solution, a postprocessing
step can be performed: The shock particle is replaced by a discontinuity, represented
by two particles $(\bar{x}_2,u_1),(\bar{x}_2,u_3)$, with their position $\bar{x}_2$
chosen, such that area is conserved exactly. This step is harmless, since an immediate
particle merge would recover the original configuration. As Fig.~\ref{fig:shock_location}
illustrates, this postprocessing locates the shock with second order accuracy.
The numerical results in Sect.~\ref{subsec:numerics_convergence_error} second this.
\end{rem}

\begin{figure}
\centering
\begin{minipage}[t]{.9\textwidth}
\centering
\includegraphics[width=0.5\textwidth]{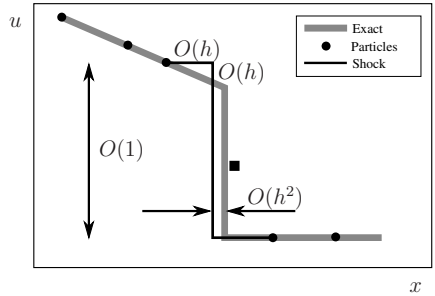}
\caption{The location of the reconstructed shock is second order
accurate in $h$. The square particle is a ``shock particle''}
\label{fig:shock_location}
\end{minipage}
\end{figure}

\begin{rem}[Shock Speed]
\label{rem:shock_speed}
Since the particle method locates shocks with second order accuracy
(Lemma~\ref{rem:shock_location}), the average shock speed is recovered with the same
accuracy. However, the instantaneous speed of a reconstructed shock is only a first
order accurate approximation to the true shock speed. The reason is that the shock
position is chosen according to a locally constant solution. This yields an $O(h)$
error in the local shock height, and thus an $O(h)$ error in the shock speed,
according to the Rankine-Hugoniot condition \cite{Evans1998}. In addition, at each
particle management, the position of the reconstructed shock jumps a distance of order $O(h^2)$.
Note that Riemann problems are solved \emph{exactly} by our method.
\end{rem}

\begin{rem}[Quadratic flux function]
\label{rem:quadratic_flux}
The method is particularly efficient for quadratic flux functions.
In this case the interpolation \eqref{eq:interpolation_function} between two
points is a straight line, and the average \eqref{eq:nonlinear_average} is the
arithmetic mean $a(u_1,u_2) = \frac{u_1+u_2}{2}$. Thus, the function values for
particle management can be computed explicitly.
\end{rem}
\begin{rem}[Computational Cost]
An interesting aspect arises in the computational cost of the method, when counting
evaluations of the flux function $f$ and its derivatives $f'$, $f''$. Particle movement
does not involve any evaluations, since the characteristic velocity of each particle
$f'(u_i)$ does not change. Consider a solution on $t\in [0,1]$ with a bounded number
of shocks, to be approximated by $O(n)$ particles. Every particle merge and insertion
requires $O(1)$ evaluations. After each time increment \eqref{eq:time_step}, $O(1)$
management steps are required. The total number of time increment steps is $O(n)$.
Thus, the total cost is $O(n)$ evaluations, as opposed to $O(n^2)$ evaluations for
Godunov methods. Note that this aspect is only apparent if evaluations of
$f'$, $f''$ are expensive, since the total number of operations is still $O(n^2)$.
\end{rem}

%=============================================================================================
\section{Sampling of the Initial Data}
\label{sec:sampling_initial_data}
%=============================================================================================
When no source terms are present, the method has two sources of error: the initial sampling,
and the merging of particles which contributes to the error in the neighborhood of shocks.
Since, after the initial sampling and away from shocks, the solution is evolved exactly,
it is natural to look for ways to reduce the error due to the initial sampling. 
In some applications, the initial function $u_0$ may be representable exactly by the
interpolation \eqref{eq:interpolation_function}. In other cases, it has to be
approximated. In Sect.~\ref{subsec:sampling_initial_data_convergence}, we show how
well $u_0$ can be approximated, as more and more particles are used.
In Sect.~\ref{subsec:sampling_initial_data_adaptivity}, we outline a strategy of
initializing a given number of particles in order to obtain a good approximation.

%---------------------------------------------------------------------------------------------
\subsection{Error Convergence}
\label{subsec:sampling_initial_data_convergence}
%---------------------------------------------------------------------------------------------
\begin{lem}
\label{lem:sampling_local}
Consider a smooth function $w(x)$ on an interval $x\in\brk{-\frac{h}{2},\frac{h}{2}}$,
with $|f''(\bar{w})|\ge C>0\;\forall \bar{w}\in\brk{w(-\frac{h}{2}),w(\frac{h}{2})}$.
Let $v(x)$ denote the interpolant \eqref{eq:interpolation_function} between
$-\frac{h}{2}$ and $\frac{h}{2}$. Then $|v(x)-w(x)| = O(h^2)$, and
$\int_{-h/2}^{h/2}|v(x)-w(x)|\ud{x} = O(h^3)$ i.e.~the approximation is second-order
accurate in both $L^\infty$ and $L^1$ norms\footnote{The power 3 in the order of the
integral is needed so that the global $L^1$ error, which results from adding up the
errors from all the intervals, is second-order.}.
\end{lem}
\begin{proof}
Substituting a Taylor expansion
$w(x) = w_0+w_0'x+\tfrac{1}{2}w_0''x^2+O(x^3)$
into $f'$ yields
\begin{equation}
f'(w(x)) = f'(w_0)+f''(w_0)w_0'x
+\tfrac{1}{2}\prn{f''(w_0)w_0''+f'''(w_0)w_0'^2} x^2+O(x^3)\;.
\label{eq:taylor_f_u}
\end{equation}
Using \eqref{eq:taylor_f_u} in \eqref{eq:interpolation_function}
yields for the interpolation $v(x)$
\begin{align}
f'(v(x))
=& f'(w(-\tfrac{h}{2}))
+\prn{f'(w(\tfrac{h}{2}))-f'(w(-\tfrac{h}{2}))}
\frac{x+\tfrac{h}{2}}{h} \nonumber \\
=& f'(w_0)-\tfrac{1}{2}f''(w_0)w_0'h
+\tfrac{1}{8}\prn{f''(w_0)w_0''+f'''(w_0)w_0'^2}h^2+O(h^3) \nonumber \\
&+ \prn{f''(w_0)w'_0h+O(h^3)}
\prn{\tfrac{x}{h}+\tfrac{1}{2}} \nonumber \\
=& f'(w_0)+f''(w_0)w'_0x
+\tfrac{1}{8}\prn{f''(w_0)w_0''+f'''(w_0)w_0'^2}h^2+O(h^3)\;.
\label{eq:taylor_f_v}
\end{align}
Comparing \eqref{eq:taylor_f_u} and \eqref{eq:taylor_f_v} yields
\begin{equation}
f'(w(x))-f'(v(x)) = \frac{1}{2}\prn{f''(w_0)w_0''+f'''(w_0)w_0'^2}
\prn{x^2-\frac14h^2}+O(h^3)\;.
\end{equation}
Using the Mean Value Theorem, we obtain the estimate
\begin{equation}
w(x)-v(x) = \frac12\frac{\prn{f''(w_0)w_0''+f'''(w_0)w_0'^2}\prn{x^2-\frac14h^2}}
{f''(\bar w)}+O(h^3)
\end{equation}
for some function value $\bar w$ between $v$ and $w$ (or $w(-\tfrac h2)$ and
$w(\tfrac h2)$). Since the denominator is bounded from below by $C$, we have our result
for the $L^\infty$ norm. The $L^1$ error of the interpolation in the interval satisfies
\begin{equation}
\label{eq:leading_order_l1_error}
  \lim_{h\rightarrow 0}\frac{\int_{-h/2}^{h/2}|w(x)-v(x)|\ud{x}}{h^3}
= \frac1{12}\frac{\prn{f''(w_0)w_0''+f'''(w_0)w_0'^2}}{f''(w_0)}\;.
\end{equation}
Thus completing the proof.
\end{proof}
The formula for the error ``density'' is used for our adaptive
sampling in the next section. We therefore name this density $e$:
\begin{equation}
e(x)=\frac1{12}\frac{\prn{f''(w(x))w''(x)+f'''(w(x))w'^2(x)}}{f''(w(x))}\;.
\end{equation}
Non-smooth functions can be approximated as well, if the discontinuities are known:
\begin{thm}
\label{thm:sampling_global}
Consider a piecewise smooth function $u_0(x)$ with finitely many discontinuities
(at known locations). Assume further that $|f''(u_0(x))|\ge C>0\ \forall x$.
Then $u_0$ can be approximated with second-order accuracy, using the
interpolations \eqref{eq:interpolation_function}.
\end{thm}
\begin{proof}
First, represent each discontinuity in $u_0$ exactly, using two particles. This
consumes only a finite number of points, thus the asymptotic behavior is not influenced.
Second, place the remaining particles equidistantly at $(x_i,u_0(x_i))$.
Since the jumps are represented exactly, the maximum error is second
order by Lemma~\ref{lem:sampling_local}.
\end{proof}
For non-convex flux functions, presented in Sect.~\ref{sec:inflection_points},
the flux function can have inflection points at particles. In an interval bounded by
an inflection particle, the second order accuracy is, in general, lost. First order
accuracy is guaranteed by the following
\begin{lem}
Consider a smooth function $w(x)$ on an interval $x\in\brk{0,h}$,
with $|f''(\bar{w})|>0\;\forall \bar{w}\in\prn{w(0),w(h)}$ (so that the interpolation
\eqref{eq:interpolation_function} exists).
Then the interpolation $v(x)$ between $0$ and $h$, given by
\eqref{eq:interpolation_function}, is at least first order accurate.
\end{lem}
\begin{proof}
The interpolation $v(x)$ is monotonous, hence one can bound
\begin{equation*}
|v(x)-w(x)|\le |w_{\text{min}}-w_{\text{max}}|\le Ch\;,
\end{equation*}
where $C\ge \max_{x\in [0,h]}|w'(x)|$ and $w_{\text{min}}$, and $w_{\text{max}}$ are
the minimum and maximum values that $w(x)$ attains over the interval.
\end{proof}
\begin{rem}
\label{rem:adaptive:sampling:inflection}
If the initial function is such that the flux crosses an inflection point, the error
attains the form $\|v-w\|_{L^1}\sim a h^2-b\log(h) h^2$ as $h\to 0$. The local $L^1$
error in a single interval abutting the inflection point is of order $O(h^2)$.
The remaining intervals add a total $L^1$ error of order $O(h^2\log(h))$. Hence, the
approximation is less than second order, but greater than first order.
\end{rem}

% ---------------------------------------------------------------------------------------------
\subsection{Adaptive Sampling}
\label{subsec:sampling_initial_data_adaptivity}
%---------------------------------------------------------------------------------------------
Due to Thm.~\ref{thm:sampling_global}, the initial data can be approximated with second
order accuracy using equidistantly spaced points. Yet, for a fixed number of points, a
non-equidistant spacing can yield a better approximation. The presented particle method
is designed for non-equidistant points. Hence, adaptive sampling strategies can be easily
used.

To minimize the error for a given number of points we use a particle density proportional to
$e^{-\frac{1}{2}}(x)$. For further information on this see \cite{Devroye1986}, for example. 
We define the integral
\begin{equation*}
E(x) = \int_0^x e^\frac{1}{2}(\xi)\ud{\xi}\;.
\end{equation*}
and sample $n+1$ points at positions
\begin{equation*}
x_i = E^{-1}\prn{E(L)\frac{i}{n}}\;.
\end{equation*} 
In the example presented in Fig.~\ref{fig:error}, this type of adaptive sampling is shown
to reduce the initial error by a factor of about 2.

%=============================================================================================
\section{Entropy}
\label{sec:entropy}
%=============================================================================================
We have argued in Sect.~\ref{subsec:interpolation} that due to the constructed
interpolation the particle method naturally distinguishes shocks from rarefaction fans.
In this section, we show that the method in fact satisfies the entropy condition
\begin{equation}
\eta(u)_t+q(u)\le 0
\label{eq:entropy_condition}
\end{equation}
for a convex entropy function $\eta$, if the shocks are resolved well enough during a merge
step. The following lemma considers the Kru\v zkov entropy pair $\eta_k(u) = \abs{u-k}$,
$q_k(u) = \sign(u-k)(f(u)-f(k))$. Holden et al.~\cite[Chap.~2.1]{HoldenRisebro2002}
show that if \eqref{eq:entropy_condition} is satisfied for $\eta_k, q_k$ (for all $k$),
then it is satisfied for any convex entropy function. Using the
interpolation \eqref{eq:interpolation_function} we show that the numerical solution
obtained by the particle method satisfies this condition. 
\begin{lem}[entropy for merging]
\label{lem:entropy_merging}
Let \mbox{$x_1<x_2=x_3<x_4$} be the locations of four particles, with particles
$2$ and $3$ to be merged, and \mbox{$f''>0$},\footnote{For the case $f''<0$, all
following inequality signs must be reversed.}
i.e.~\mbox{$u_2>u_3$}.
If the value $u_{23}$ resulting from the merge satisfies \mbox{$u_1\ge u_{23}\ge u_4$},
then the Kru\v zkov entropy $\int\abs{u-k}\ud{x}$ does not increase due to the merge.
\end{lem}
\begin{proof}
Consider the segment $[x_1,x_4]$. Let $u(x)$ and $\hat u(x)$ denote the interpolating
function before and after the merge, respectively. The interpolating function $u$ is
monotone in the value of its endpoints, thus $u(x)\le\hat u(x)$ for $x\in [x_2,x_4]$,
and $u(x)\ge\hat u(x)$ for $x\in [x_1,x_2]$. The function 
\begin{equation*}
 I_+(x)=\begin{cases}
x & x>0\\
0 &x\le 0
\end{cases}
\end{equation*}
can be used to write $\abs{x}=x+2\, I_+(-x)$. 
We identify two possible cases: $k\le u_{23}$ and $k\ge u_{23}$.
In the first case, $k\le u_{23}$, we write
\begin{align*}
\int_{x_1}^{x_4}\! \abs{u\!-\!k}\ud{x}
&= \int_{x_1}^{x_4}\!\prn{u\!-\!k}\ud{x}
 +2\int_{x_1}^{x_4} I_+(k\!-\!u)\ud{x}.
\intertext{Due to the definition of $\hat u$ we have}
&= \int_{x_1}^{x_4}\!\prn{\hat u\!-\!k}\ud{x}
 +2\int_{x_1}^{x_2} I_+(k\!-\!u)\ud{x}+2\int_{x_2}^{x_4} I_+(k\!-\!u)\ud{x}.
\intertext{Since $k\le u$ on $[x_1,x_2]$ and that $ I_+(u)$ is
  non-decreasing, we get}
&\ge \int_{x_1}^{x_4}\!\prn{\hat u\!-\!k}\ud{x}
 +0+2\int_{x_2}^{x_4} I_+(k\!-\!\hat u)\ud{x}.
\intertext{Using $k\le\hat u$ on $[x_1, x_2]$ we replace the zero with a different integral}
&= \int_{x_1}^{x_4}\!\prn{\hat u\!-\!k}\ud{x}
 +2\int_{x_1}^{x_2} I_+(k\!-\!\hat
 u)\ud{x}+2\int_{x_2}^{x_4} I_+(k\!-\!\hat u)\ud{x}\\
&= \int_{x_1}^{x_4}\!\prn{\hat u\!-\!k}\ud{x}
 +2\int_{x_1}^{x_4} I_+(k\!-\!\hat u)\ud{x}=\int_{x_1}^{x_4}\! \abs{\hat u\!-\!k}\ud{x}.
\end{align*}
The other option we have is $k\ge u_{23}$.
In this case the proof is quite similar, but we start with $\abs{k-u}$
instead, and remember that on $[x_1,x_2]$ we have that $u\ge\hat u$:
\begin{align*}
\int_{x_1}^{x_4}\! \abs{k\!-\!u}\ud{x}
&= \int_{x_1}^{x_4}\!\prn{k\!-\!u}\ud{x}
 +2\int_{x_1}^{x_4} I_+(u\!-\!k)\ud{x}\\
&= \int_{x_1}^{x_4}\!\prn{k\!-\!\hat u}\ud{x}
 +2\int_{x_1}^{x_2} I_+(u\!-\!k)\ud{x}+2\int_{x_2}^{x_4} I_+(u\!-\!k)\ud{x}\\
&\ge \int_{x_1}^{x_4}\!\prn{k\!-\!\hat u}\ud{x}
 +2\int_{x_1}^{x_2} I_+(\hat u \!-\!k)\ud{x}+0\\
&= \int_{x_1}^{x_4}\!\prn{k\!-\!\hat u}\ud{x}
 +2\int_{x_1}^{x_2} I_+(\hat u\!-\!k)\ud{x}+2\int_{x_2}^{x_4} I_+(\hat u\!-\!k)\ud{x}\\ 
&= \int_{x_1}^{x_4}\!\prn{k\!-\!\hat u}\ud{x}
 +2\int_{x_1}^{x_4} I_+(\!\hat u\!-\!k)\ud{x}=\int_{x_1}^{x_4}\!
 \abs{k\!-\!\hat u}\ud{x}.
\end{align*}
This ends the proof.
\end{proof}
The assumption of Lemma~\ref{lem:entropy_merging} implies that shocks must be reasonably
well resolved before the points defining it are merged. It is satisfied if the points to
the left and right of a shock points are not too far. The condition can be ensured by an
\emph{entropy fix}: A merge is rejected \emph{a posteriori} if the resolution condition
is not satisfied. Then, points are inserted near the shock, and the merge is re-attempted.
\begin{rem}
With the entropy fix, a merge does not necessarily reduce the number of points. Based
on numerical evidence, we conjecture that the statement of Thm.~\ref{thm:arbitrary_times}
remains valid, although its proof cannot be transferred in a straightforward fashion.
\end{rem}
\begin{thm}
\label{thm:entropy}
The presented particle method yields entropy solutions.
\end{thm}
\begin{proof}
During the characteristic movement of the points, the entropy is constant. This is due to
Cor.~\ref{cor:exact_solution} which tells that the interpolation is a classical solution
to the conservation law. Particle insertion does not change the interpolation, thus it 
does not change the entropy. Merging does not increase the entropy if the conditions of
Lemma~\ref{lem:entropy_merging} are satisfied.
\end{proof}

%=============================================================================================
\section{Non-Convex Flux Functions}
\label{sec:inflection_points}
%=============================================================================================
So far we have only considered flux functions without inflection points (i.e.~$f''$
always has the same sign) on the range of function values. In this section, we generalize
our method for flux functions $f$ where $f''$ has a finite number of zero crossings
\mbox{$u^*_1<\dots<u^*_k$}.
Between two successive points \mbox{$u\in [u^*_i,u^*_{i+1}]$} the flux function is either
convex or concave. We impose the following requirement for any set of particles: Between
any two particles for which $f''$ has opposite sign, there must be an
\emph{inflection particle} $(x,u_i^*)$. Thus, between two neighboring particles, $f$ never
has an inflection point, and the fundamental ideas from the previous sections transfer.
In particular, the characteristic movement of particles is unaffected, and the
interpolation between two particles remains uniquely defined
by \eqref{eq:interpolation_function}. It has infinite slope at inflection points
(since $f''(u_i^*) = 0$), but this is mostly harmless.
However, two complications arise. First, every proof that relies on having a lower bound
on $f''$ does not transfer easily. Second, merging particles when an inflection particle
is involved requires a special treatment. The standard approach, as presented
in Sect.~\ref{subsec:particle_management}, removes two colliding points and
replaces them with a point of a different function value. If an inflection particle is
involved in a collision, points must be merged in a different way so that an inflection
particle remains.

\begin{figure}
\centering
\begin{minipage}[t]{.32\textwidth}
\centering
\includegraphics[width=.99\textwidth]{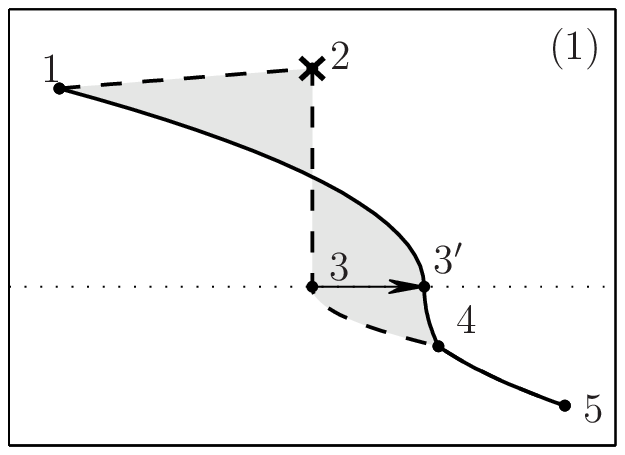}
\end{minipage}
\hfill
\begin{minipage}[t]{.32\textwidth}
\centering
\includegraphics[width=.99\textwidth]{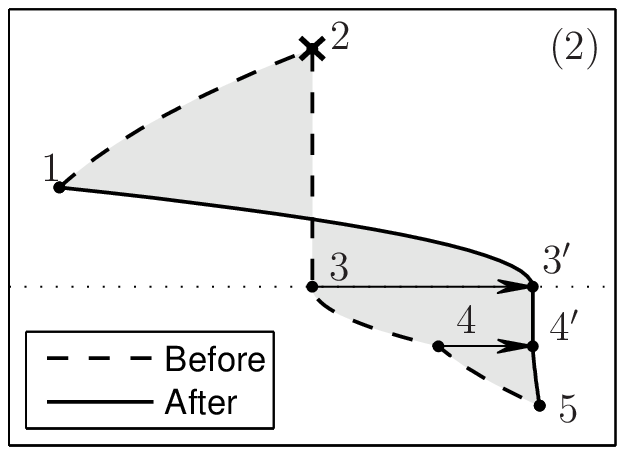}
\end{minipage}
\hfill
\begin{minipage}[t]{.32\textwidth}
\centering
\includegraphics[width=.99\textwidth]{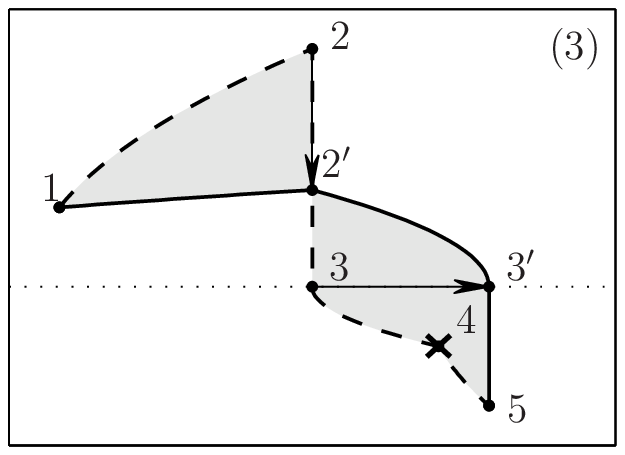}
\end{minipage}
\caption{Particle management around an inflection particle ($f''(u_3)=0$) results in one
of three possible configurations. Each configuration allows for more area under the
function than the previous one. Here we see three archetypal particle configurations
that result.}
\label{fig:particle_management_inflection}
\end{figure}

We present one such special merge for dealing with a single inflection point (we do not
consider here the interaction of two inflection points). Also, for simplicity, we
consider a collision with identical point positions. Since the inflection particle must
remain (although its position may change), we consider five neighboring particles and
not four as before. Let $(x_i,u_i), i=1,\ldots,5$ be these particles so that $x_2=x_3$,
$f''(u_3)=0$, and (WLOG) $f'''>0$, i.e.~the inflection particle is the slowest. The
other cases are simple symmetries of this situation. 
The special merge consists of three different attempts to find the new particle configuration. 
The first two attempts may fail to provide a solution, in which case the next is attempted.
\begin{enumerate}
\item
Remove particle 2 and increase $x_3$ such that area is preserved.
Accept, if $x_3$ is not increased beyond $x_4$.
\item
Remove particle 2, set $x_3=x_4$ and increase both such that area is preserved.
Accept, if $x_3$ and $x_4$ are not increased beyond $x_5$.
\item
Remove particle 4, set $x_3=x_5$ and lower $u_2$ such that area is preserved.
\end{enumerate}
\begin{thm}
\label{thm:non-convex:merge}
One of the three attempts listed above will succeed.
\end{thm}
\begin{proof}
Following from continuity and monotonicity of the average function $a(\cdot,\cdot)$,
the three steps provide a continuous, monotonous increase in area. In the first attempt,
the smallest area is achieved with $x_3$ unchanged. This area is necessarily smaller
than the original area (since one can also get here by lowering $u_2$ to $u_3$). 
The area increases as $x_3$ is increased. The configuration with $x_3=x_4$ has the
maximum area for the first attempt, and the minimum area for the second. Again, the area
increases as $x_3=x_4$ increase. 
The configuration with $x_3=x_4=x_5$ has the maximum area for the second attempt, and
the minimum for the third. Area increases as the new value of $u_2$ increases,
and achieves its maximum value for an unchanged $u_2$. 
This area is necessarily larger than the original area. 
Consequently, one of the attempts must succeed.
\end{proof}
\begin{rem}
The resulting configuration may involve a new discontinuity
(since $x_3=x_4$ or $x_3=x_5$). However, this is not a shock, but a rarefaction,
since the particles will move away from each other. Consequently, these
particles should \emph{not} be merged.
\end{rem}

The five-point particle management guarantees that in each merging step one particle is
removed, as used in Thm.~\ref{thm:arbitrary_times}.
In Sect.~\ref{subsec:numerics_non_convex_flux}, numerical results on the Buckley-Leverett
equation are presented. Since each of the three attempts covers a non-overlapping range of
areas, the resulting configuration is independent of the order in which they are attempted.

%=============================================================================================
\section{Sources}
\label{sec:sources}
%=============================================================================================
An important extension of the conservation law \eqref{eq:conservation_law_space_indep}
is to allow a source term in the right hand side. This can be a function of $x$, $t$,
the function value $u$, and in principle also of derivatives $u_x$, $u_{xx}$, etc.
In the current work we consider the simple balance law
\begin{equation}
u_t+f(u)_x = g(x,u)\;.
\label{eq:conservation_source}
\end{equation}
The method of characteristics \cite{Evans1998} yields an evolution for each particle
\begin{equation}
\begin{cases}
\dot x = f'(u) \\
\dot u = g(x,u)
\end{cases}
\label{eq:characteristic_equation_sources}
\end{equation}
With sources, the equation ceases to have exact conservation properties. Consequently, the
interpolation derived in Sect.~\ref{subsec:interpolation} is no longer a solution.
While in special cases more complicated interpolation functions could be defined
(depending on both $f$ and $g$), here we construct an approximate method that is more
general. Assume that the advection dominates over the source, which is the case in many
applications. Thus, the interpolation and particle management are based solely on the
flux function $f$.

The source $g$ results in a vertical movement of the particles during their Lagrangian
evolution. While in the absence of sources the next time of a particle merge can be
computed a priori, now we solve the particle
evolution \eqref{eq:characteristic_equation_sources} numerically, for instance by an
explicit Runge-Kutta scheme. Merging takes place when two particles are too close
(see Rem.~\ref{rem:merging_robust}).
In Sect.~\ref{subsec:numerics_source_terms}, we present numerical results.
\begin{rem}
The balance law \eqref{eq:conservation_source} is solved correctly at characteristic points.
Particle management, however, is based on an ``incorrect'' interpolation, since the source
is neglected for the definition of area. The numerical results in
Sect.~\ref{subsec:numerics_source_terms} indicate that this does not cause problems for
merging particles. However, inserting particles into large gaps may lead to significant
misplacements, when the source is ``active''. Thus, with sources, insertion should
either be avoided completely, or particles be adaptively refined. We shall address the
important aspect of adaptivity in future work.
\end{rem}
The presented approach incorporates sources directly into the characteristic equations.
An alternative approach is operator splitting: First move particles neglecting the source,
then correct function values according to the source. While the characteristic method is
more precise, the splitting approach is more general. In particular, it can deal with
source terms that involve derivatives of $u$.

%=============================================================================================
\section{Numerical Results}
\label{sec:numerical_results}
%=============================================================================================
The presented particle method is applied to various examples. In all cases, the ``exact''
reference solution is obtained or verified by a high resolution CLAWPACK \cite{Clawpack}
computation. We compare the accuracy of the particle method with numerical solutions
obtained by CLAWPACK, considering similar resolutions. By construction, the particle
method does not keep a fixed resolution. To compare resolutions we use the same number of
particles to initialize the particle method as the number of cells in the corresponding
CLAWPACK run. By keeping $d_\text{max} = \frac{4}{3}h$, we find that the number of
particles remains more-or-less constant throughout the computation. Shocks are located via
post-processing before the error is measured.

In Sect.~\ref{subsec:numerics_convergence_error}, the evolution and the formation of shocks
of smooth initial data under a convex flux function are considered. The convergence error
before and after the occurrence of shocks is investigated numerically.
In Sect.~\ref{subsec:numerics_non_convex_flux}, as an example of a non-convex flux function,
the Buckley-Leverett equation is considered, and in
Sect.~\ref{subsec:numerics_source_terms}, Burgers' equation with a source is simulated.
The source code and all presented examples can be found on the 
\texttt{particleclaw} web page \cite{Particleclaw}.

\begin{figure}
\centering
\begin{minipage}[t]{.32\textwidth}
\centering
\includegraphics[width=0.99\textwidth]{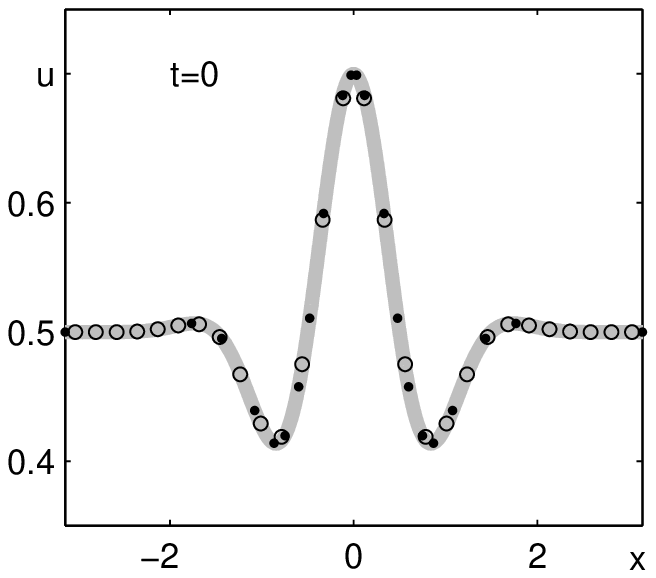}
\end{minipage}
\hfill
\begin{minipage}[t]{.32\textwidth}
\centering
\includegraphics[width=0.99\textwidth]{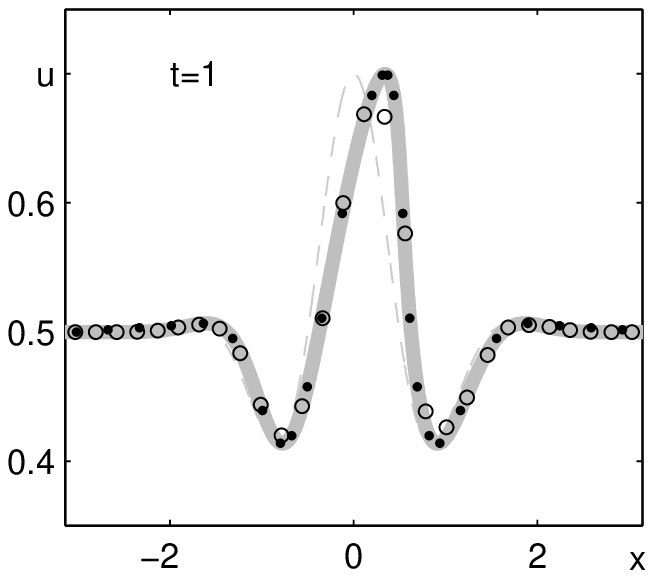}
\end{minipage}
\hfill
\begin{minipage}[t]{.32\textwidth}
\centering
\includegraphics[width=0.99\textwidth]{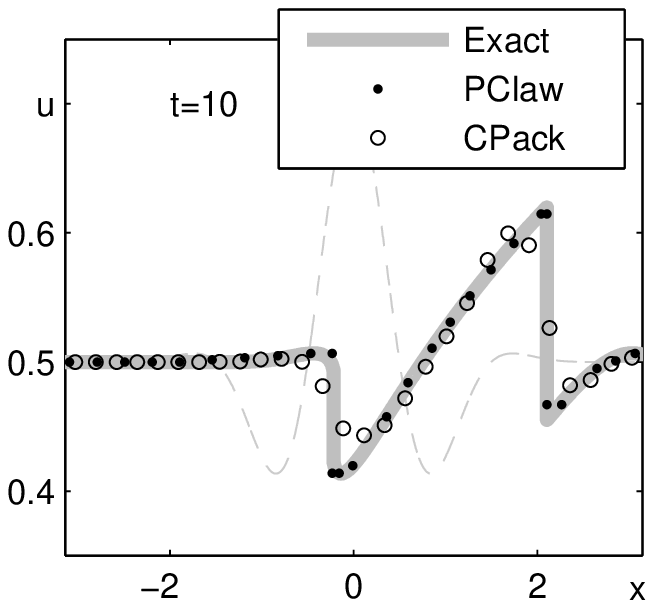}
\end{minipage}
\caption{The particle method for $f(u)=\tfrac{1}{4}u^4$ before and after shocks arise. 
The gray solid line is the hi-resolution solution from CLAWPACK and
the dashed line is the initial condition.}
\label{fig:u4}
\end{figure}

%---------------------------------------------------------------------------------------------
\subsection{Convergence Error}
\label{subsec:numerics_convergence_error}
%---------------------------------------------------------------------------------------------
Figure~\ref{fig:u4} shows the smooth initial function
$u_0(x) = 0.5+0.2\exp\!\prn{-x^2}\cos(\pi x)$, and its time evolution under the flux
function $f(u)=\frac{1}{4}u^4$. Initially, we sample points on the function $u_0$.
At time $t=1$, the solution is still smooth, thus the particles lie exactly on the
solution. By the time $t=10$, a shock has emerged and interacted with a rarefaction.
Although the numerical solution uses only a few points, it represents the
true solution well.

{}From this example, the numerical accuracy of the particle method is extracted.
For a sequence of particle densities, the initial data are sampled twice: equally spaced
and adaptively. The particle method is applied with post-processing, as described in
Rem.~\ref{rem:shock_location}. The error is measured in the true $L^1$ norm for function,
which is possible due to the interpolation \eqref{eq:interpolation_function}.
Figure~\ref{fig:error} shows results for initial sampling error, and error after
a time evolution. Initially ($t=0$), the approximation is second order accurate for
both sampling strategies (see Thm.~\ref{thm:sampling_global}). The advantage of the
adaptive sampling is evident from the lower error that it creates in the interpolation.

After shocks have occurred ($t=10$), the approximate solution without locating shocks is
only first order accurate, since at any shock an error of the order height$\times$width
of the shock is made. However, the post-processing step recovers the second order accuracy.
Hence, the particle method is second order accurate, even at locating shocks. One also sees
that the advantage gained initially from the adaptive sampling is nearly lost at $t=10$. 

For this example, CLAWPACK yields results of similar accuracy, as shown in
Fig.~\ref{fig:error:compare}. Since CLAWPACK is a method for calculating cell averages,
we cannot find the true $L^1$ error. Instead, given a coarse-grid calculation and a
fine-grid reference solution, we calculate the error in the area (function value times
cell-size) for each of the coarse cells using the fine-grid solution. Adding all these
errors together gives the relevant $L^1$ error. One can see that the CLAWPACK solution
drops to first-order accuracy around shocks, which is due to numerical dissipation. To
investigate the accuracy away from shocks, we also consider the error while ignoring a
small fixed domain surrounding each shock. The same error measure is also applied to the
error calculations of our particle method. Of course, since post-processing already yields
second order accuracy, this only reduces the size of the error, and does not change the
order of convergence, as it does with CLAWPACK. 
\begin{figure}
\centering
\begin{minipage}[t]{.48\textwidth}
\centering
\includegraphics[width=.99\textwidth]{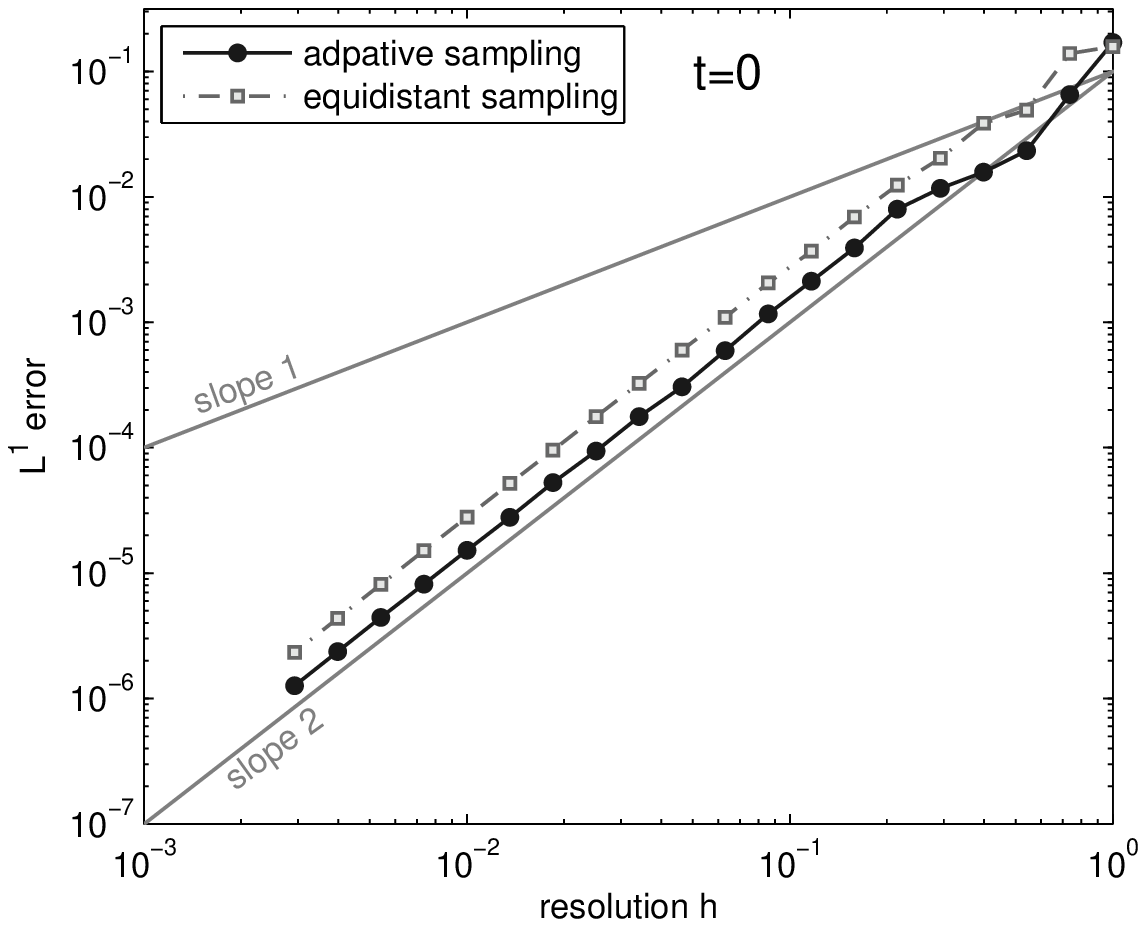}
\end{minipage}
\hfill
\begin{minipage}[t]{.48\textwidth}
\centering
\includegraphics[width=.99\textwidth]{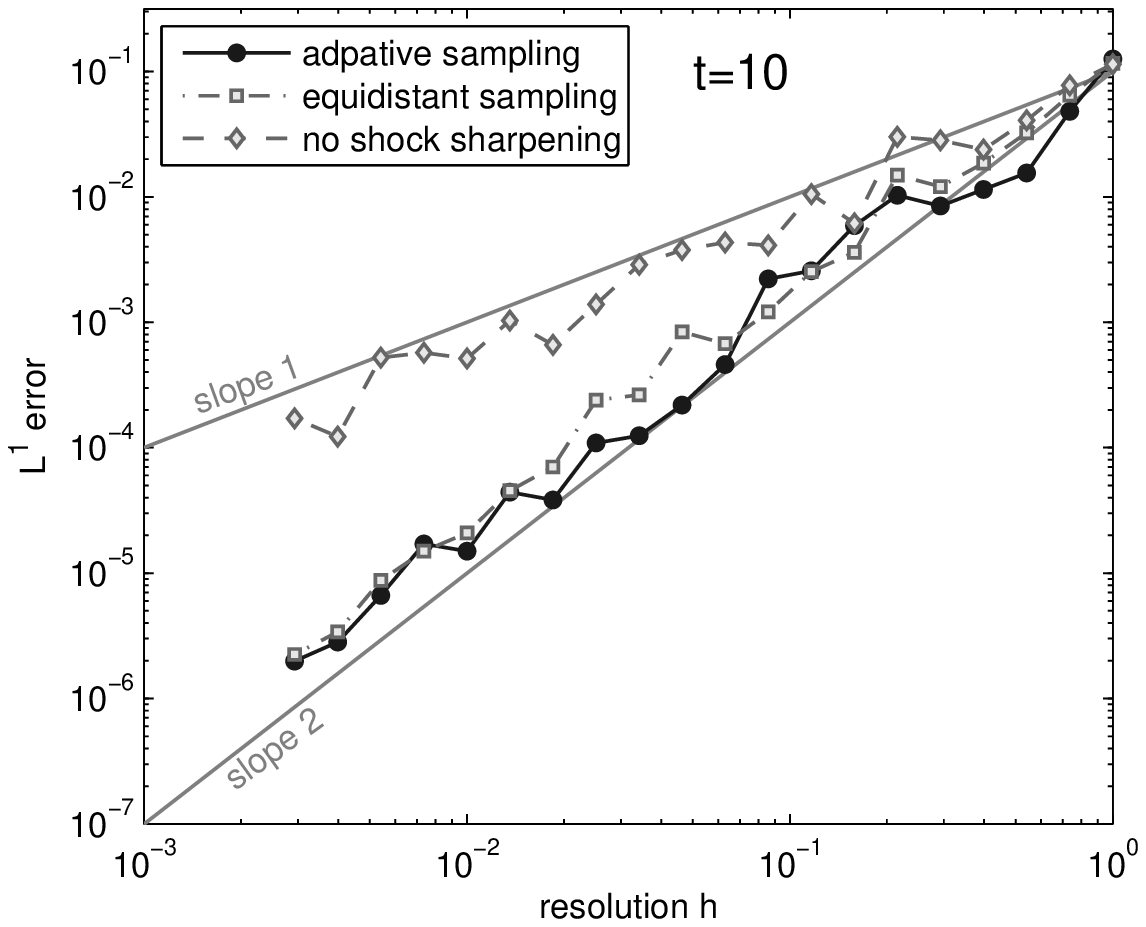}
\end{minipage}
\caption{$L^1$ Error of the particle method when solving for the flux
function $\frac{1}{4} u^4$, initially, and after a shock has been formed. The figures
compare equidistant with adaptive sampling and show the error that would result from
not sharpening the shock.}
\label{fig:error}
\end{figure}

\begin{figure}
\centering
\begin{minipage}[t]{.48\textwidth}
\centering
\includegraphics[width=.99\textwidth]{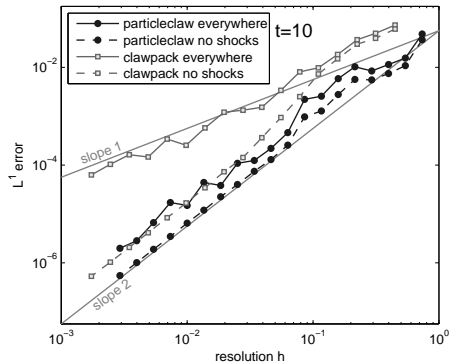}
\end{minipage}
\caption{A comparison of the errors given by  the particle method and CLAWPACK when
solving for the flux  function $\frac{1}{4}u^4$. Without removing the errors from the
shock region, CLAWPACK is only first-order accurate.}
\label{fig:error:compare}
\end{figure}
%---------------------------------------------------------------------------------------------
\subsection{Non-Convex Flux Function}
\label{subsec:numerics_non_convex_flux}
%---------------------------------------------------------------------------------------------
As an example of a non-convex flux function, we consider the Buckley-Leverett equation
\begin{equation}
\label{eq:bucklev}
u_t+\prn{f(u)}_x = 0\ \text{, with}\ f(u) = u^2/(u^2+\tfrac{1}{2}(1-u)^2)\;,
\end{equation}
which is a simple model for two-phase fluid flow in a porous medium
(see LeVeque \cite{LeVeque2002}). We consider piecewise constant initial data with a large
downward jump crossing the inflection point, and a small upward jump. The large jump
develops a shock at the bottom and a rarefaction at the top, the small jump is a pure
rarefaction. Around $t=0.2$, the two similarity solutions interact, thus lowering the
separation point between shock and rarefaction. Figure~\ref{fig:results_inflection} shows
numerical results. The solution obtained by the particle method (dots) is compared to a
second order CLAWPACK solution (circles) of similar resolution.
The particle method captures the behavior of the solution better; in particular, the
rarefaction is represented very accurately. Only directly near the shock are inaccuracies
visible. The solution away from the shock is nearly unaffected by the error at the shock.

\begin{figure}
\centering
\begin{minipage}[t]{.32\textwidth}
\centering
\includegraphics[width=1.02\textwidth]{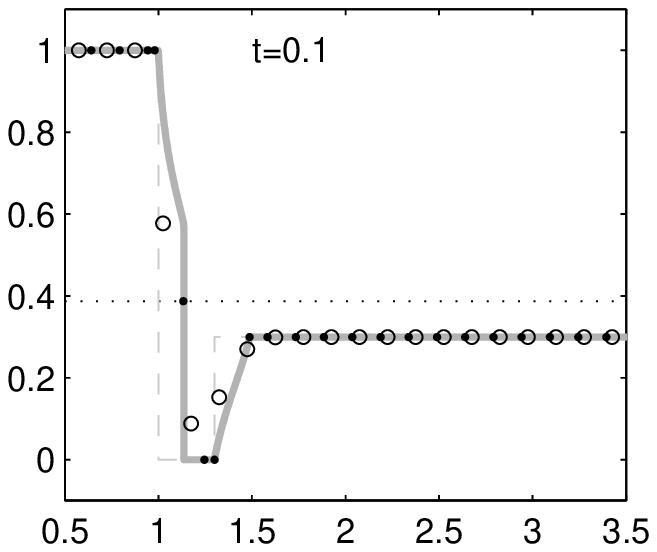}
\end{minipage}
\hfill
\begin{minipage}[t]{.32\textwidth}
\centering
\includegraphics[width=1.02\textwidth]{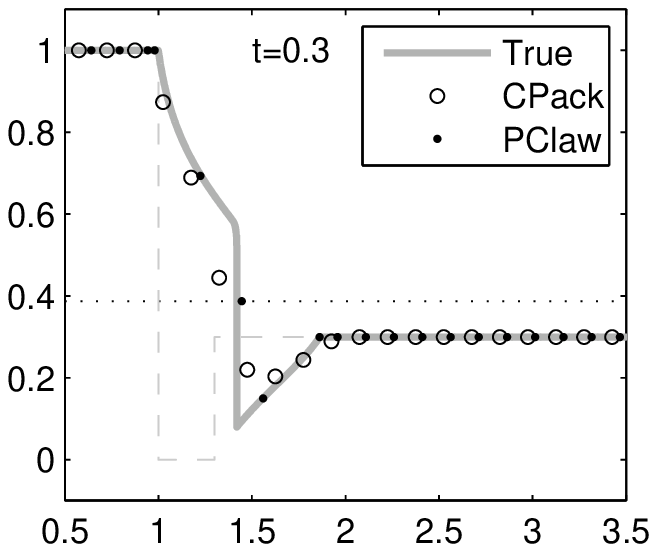}
\end{minipage}
\hfill
\begin{minipage}[t]{.32\textwidth}
\centering
\includegraphics[width=1.02\textwidth]{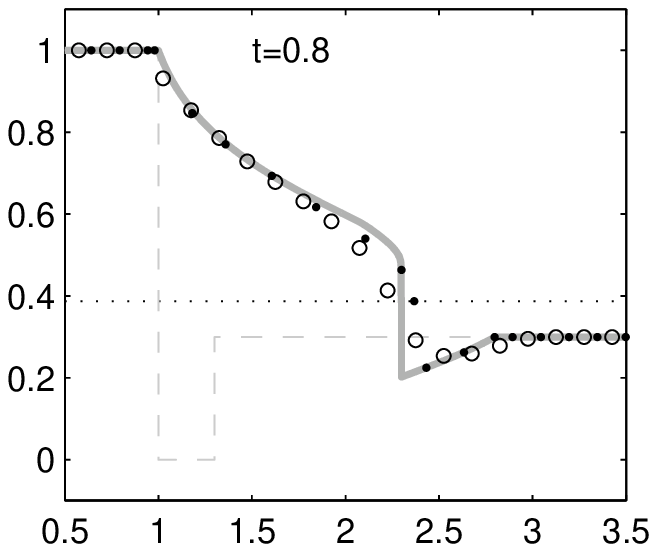}
\end{minipage}
\caption{Numerical results for the Buckley-Leverett equation at various times $t$}
\label{fig:results_inflection}
\end{figure}
\begin{figure}
\centering
\begin{minipage}[t]{.48\textwidth}
\centering
\includegraphics[width=.99\textwidth]{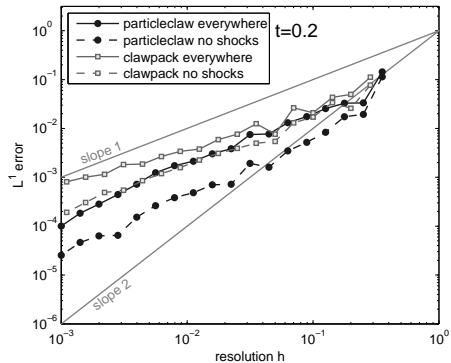}
\end{minipage}
\caption{A comparison of the errors given by  the particle method and CLAWPACK when
solving the Buckley-Leverett equation \eqref{eq:bucklev}. While the particle method
gets significantly better results than CLAWPACK (with or without the errors from
around the shock), both methods are less than second order accurate.}
\label{fig:error:compare:bucklev}
\end{figure}
Numerical results show (see Fig.~\ref{fig:error:compare:bucklev}) that both CLAWPACK and
the presented particle method do not achieve second order accuracy for this problem.
Nevertheless, the particle method has a much better accuracy than
CLAWPACK.
The drop in accuracy is, presumably, due to inflection point in the
Buckley-Leverett flux function, similar to the drop in accuracy of the
sampling outlined in Remark~\ref{rem:adaptive:sampling:inflection}.

%---------------------------------------------------------------------------------------------
\subsection{Source Terms}
\label{subsec:numerics_source_terms}
%---------------------------------------------------------------------------------------------
We consider Burgers' equation with a source
\begin{equation}
u_t+\prn{\tfrac{1}{2}u^2}_x = b'(x)u\;.
\label{eq:burgers:source}
\end{equation}
It is a simple model for shallow water flow over a bottom profile $b(x)$.
As in \cite{KarlsenMishraRisebro2008}, we consider the domain $x\in [0,10]$, and choose
\begin{equation*}
b(x) =
\begin{cases}
\cos(\pi x) & x\in [4.5,5.5] \\
0 & \text{otherwise}\;.
\end{cases}
\end{equation*}
The source term is included into the method of characteristics, as explained
in Sect.~\ref{sec:sources}. The time stepping is done by a fourth order Runge-Kutta
scheme. Figure~\ref{fig:results_source} shows the computational results. The particle
method (dots) approximates the solution significantly better than the second order
CLAWPACK scheme (circles). A particular aspect in favor of the characteristic approach
is the precise (up to the resolution of the ODE solver) recovery of the function values
after the obstacle. Since particles are
moved independently according to the characteristic equations, an accurate time
integration obtains the function values after the obstacle almost exactly, independent
of the resolution of particles. Note that an efficient treatment of the source
requires a special consideration of its discontinuities, either in the quadrature of the
source (finite volume), or in the integration of the characteristic ODE (particle scheme).

\begin{figure}
\begin{tabular}{cc}
\begin{minipage}[t]{.48\textwidth}
\includegraphics[width=.95\textwidth]{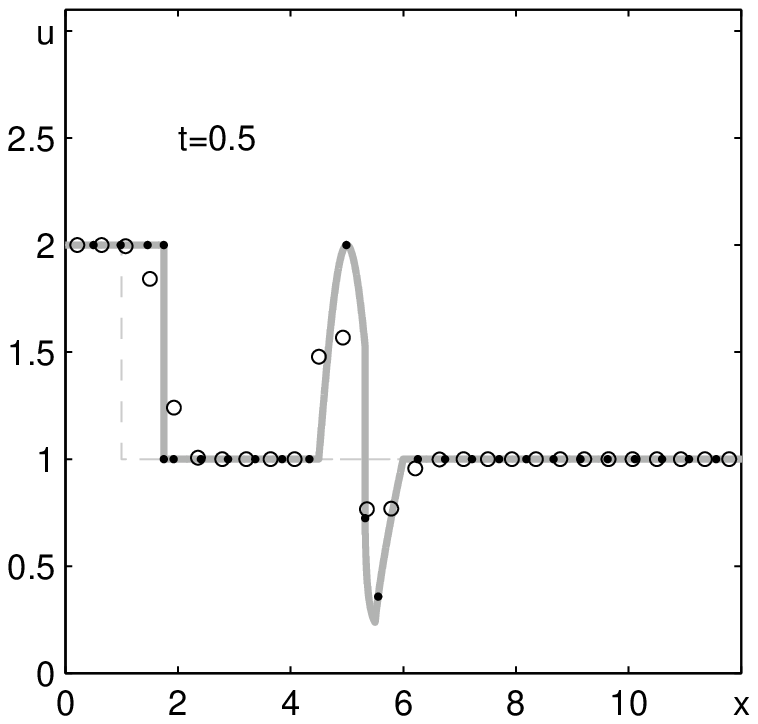}
\end{minipage}
&
\begin{minipage}[t]{.48\textwidth}
\includegraphics[width=.95\textwidth]{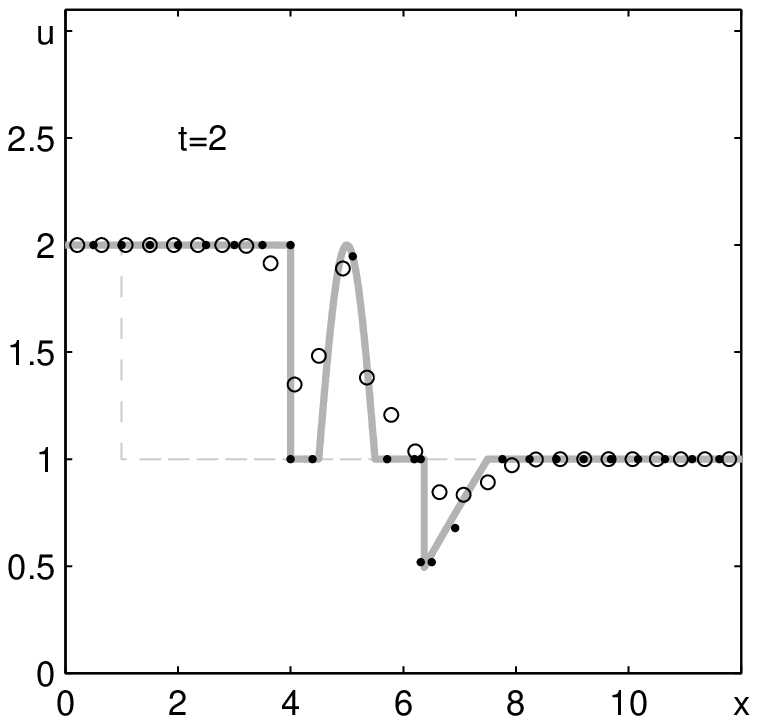}
\end{minipage}
\\
\begin{minipage}[t]{.48\textwidth}
\includegraphics[width=.95\textwidth]{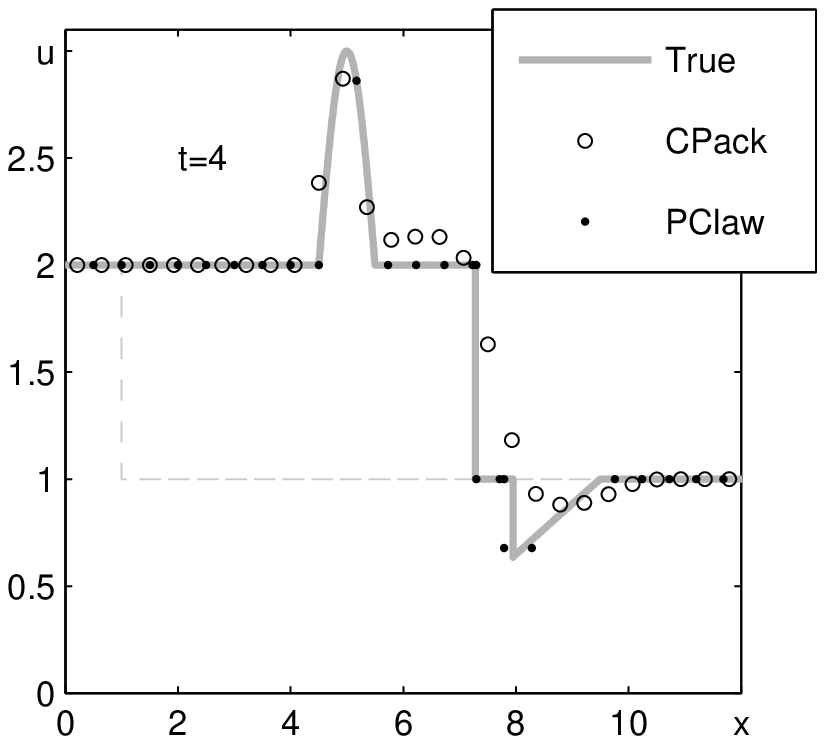}
\end{minipage}
&
\begin{minipage}[t]{.48\textwidth}
\includegraphics[width=.95\textwidth]{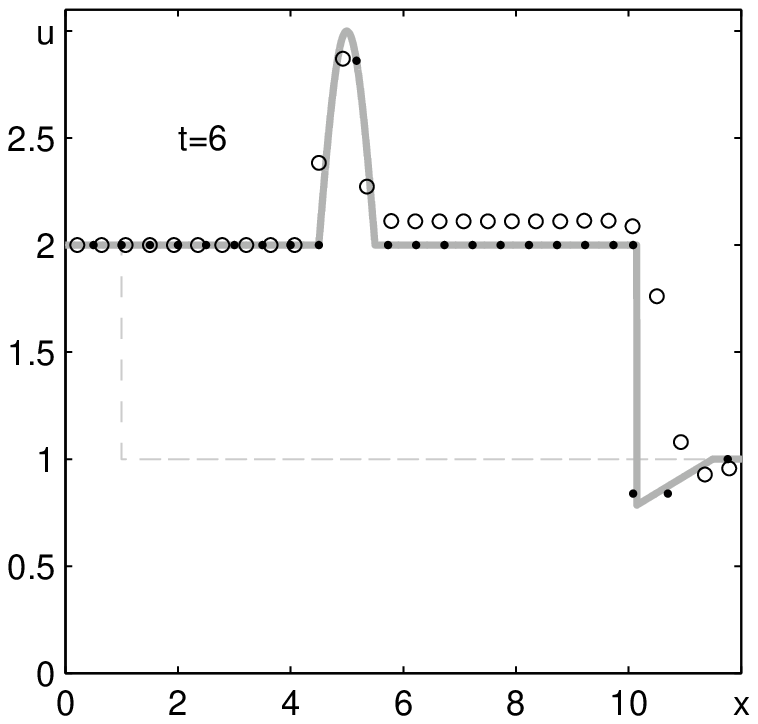}
\end{minipage}
\end{tabular}

\caption{Time evolution of Burgers' equation with a source as given
by \eqref{eq:burgers:source}. The dots, circles, gray line and dashed line are,
respectively, our particle method, CLAWPACK, the high-resolution solution and the
initial conditions.}
\label{fig:results_source}
\end{figure}

\begin{figure}
\centering
\begin{minipage}[t]{.48\textwidth}
\centering
\includegraphics[width=.99\textwidth]{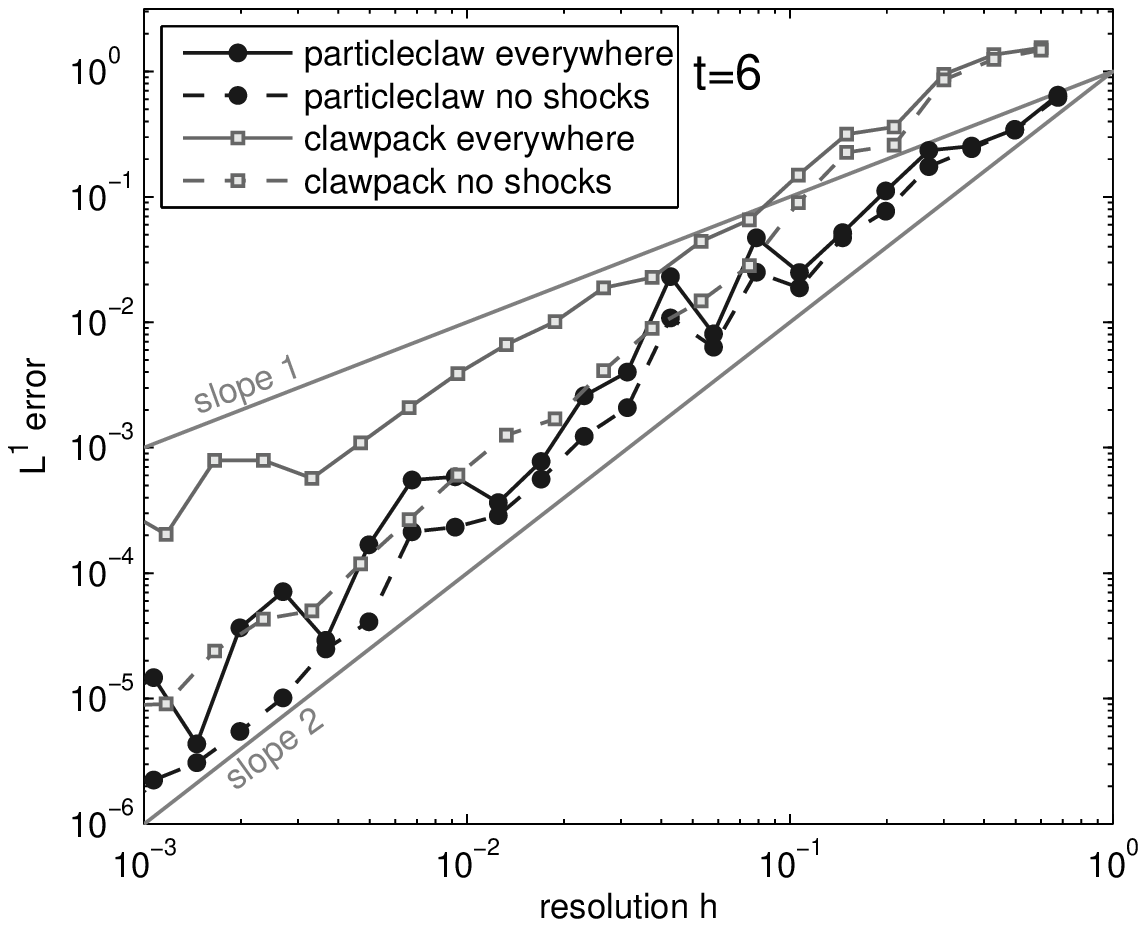}
\end{minipage}
\caption{A comparison of the errors given by  the particle method and CLAWPACK when
solving for Burgers' equation with a source term \eqref{eq:burgers:source},
at $t=6$ (after the two shocks have interacted). Here too, the particle method
provides a better accuracy than CLAWPACK.}
\label{fig:error:compare:source}
\end{figure}

%=============================================================================================
\section{Conclusions and Outlook}
\label{sec:outlook}
%=============================================================================================
We have presented a particle method that combines the method of characteristic, local
similarity solutions, and particle management to a numerical scheme for for 1D scalar
conservation laws. The method conserves area exactly. It is TVD, yet second order
accurate, even at locating shocks. It performs promisingly in various examples, as the
numerical comparisons with a second order finite volume scheme show.

The particle method is an interesting alternative to fixed grid approaches, whenever
conservation of mass is crucial, or shocks need to be located accurately. In addition,
entropy is reduced only when particles are merged, which makes the approach suited for
applications in which the evolution of mass and energy has to be reflected as precisely
as possible. Furthermore, the method yields good results when few particles are used,
in particular shocks between nearly-constant states are located well. This makes the
approach attractive whenever scalar 1D conservation laws arise as sub-problems in a
large computation, and only a few degrees of freedom can be devoted to the numerical
solution of a single sub-problem. Examples are flows in networks (e.g.~car traffic),
and PDE constrained optimization.

As a first generalization, we have included source terms in the scheme. The method,
still based on the method of characteristics, yields solutions of rather striking accuracy,
compared to classical finite volume schemes. In future work, more general source terms
will be considered, such as non-local convolutions, and terms involving derivatives of the
solution. In these cases, the method of characteristics has to be replaced by a more
general splitting approach.

Fundamental steps towards a more powerful particle method will be the generalization
to higher space dimensions and to systems of conservation laws.
Problems in multiple dimensions can be approximated by 1D problems using fractional
steps. In this sense, the particle scheme could replace classical 1D Riemann solvers
by 1D \emph{wave solvers}. However, this approach is not fully satisfactory, since due
to the required remeshing steps the benefits of a meshfree particle approach may be
lost. On the other hand, with truly meshfree approaches in 2D/3D, one has to address the
challenge that particles forming a shock need not necessarily collide. Possible remedies
are the introduction of a numerical pressure, or the tracking of an unstructured
triangular mesh. The movement of particles according to the method of characteristics
can also be interpreted as a moving mesh approach \cite{BainesHubbardJimack2005}.
Thus, ideas from this area could lead to particle strategies in higher space dimensions.

With systems, one difficulty is the presence of multiple characteristic velocities.
One approach is to choose one Lagrangian velocity, which need not be a characteristic
velocity. Coupling terms that appear in the moving frame equations are treated as source
terms for each individual equation. Alternative approaches may use exact similarity
solutions of the full system as building blocks. In this case, a single set of
particles may not suffice, since two neighboring similarity solutions may interact.

%=============================================================================================
\bibliographystyle{amsplain}
%\bibliography{seibold_references}
\bibliography{references_complete}
%=============================================================================================

\appendix
%=============================================================================================
\section{Appendix}
%=============================================================================================
The proofs of Lemmas~\ref{lem:merge:lower_bound} and~\ref{lem:merge:upper_bound} use a
short lemma:
\begin{lem}
The derivative of $a(u,v)$ with respect to either of its variables is bounded from below
and above as follows:
\begin{equation*}
\frac12\prn{\frac{\min f''}{\max f''}}^2\le
\brk{\pd{a}{u}(u,v),\pd{a}{v}(u,v)}
\le\frac12\prn{\frac{\max f''}{\min f''}}^2.
\end{equation*}
\end{lem}
\noindent Here $\max f''$ and $\min f''$ are taken over the interval $[u,\,v]$.
\begin{proof}
This lemma follows from the definition of $a$:
\begin{align*}
\pd{a}{u}(u,v)
&=\frac{f''(u) \int_u^v f''(\omega)(\omega-u)\ud{\omega}}
{\prn{\int_u^v f''(\omega)\ud{\omega}}^2} \\
&\le \prn{\frac{\max f''}{\min f''}}^2 \frac{\int_u^v \omega-u \, d\omega}{(v-u)^2}
\le \frac12\prn{\frac{\max f''}{\min f''}}^2.
\end{align*}
The other bounds (on $\pd{a}{v}$ and the lower bound) have similar proofs.
\end{proof}
\begin{proof}[of Lemma~\ref{lem:merge:lower_bound}]
WLOG we assume that $u_3\ge u_2$, and show for $u_2$.
We bound $A-B(u_2)$ from below:
\begin{align}
A-B(u_2)&=(x_3-x_2)(a(u_2,u_3)-a(u_2,u_2))+(x_4-x_3)(a(u_3,u_4)-a(u_2,u_4)) \nonumber \\
&\ge (x_3-x_2)(u_3-u_2) \min \pd{a}{v}(u,v)+(x_4-x_3)(u_3-u_2)\min \pd{a}{u}(u,v)
\nonumber \\
&\ge (x_4-x_2)(u_3-u_2)\frac12\prn{\frac{\min f''}{\max f''}}^2.
\label{eq:app_bound_below}
\end{align} 
Since we are looking for $\tilde u$ such that $B(\tilde u)=A$, the previous bound is
also a bound on $B(\tilde u)-B(u_2)$.
{}From the Mean Value Theorem we have $\xi\in\brk{\tilde u,\,u_2}$ for which
\begin{align*}
\tilde u - u_2 &= \frac{B(\tilde u)-B(u_2)}{B'(\xi)} \\
&= \frac{B(\tilde u)-B(u_2)}
{(x_2-x_1)\pd{a}{u}(\xi,u_1)+(x_4-x_3)\pd{a}{u}(\xi,u_4)+(x_3-x_2)} \\
&\ge \frac{B(\tilde u)-B(u_2)}{(x_4-x_1)\prn{\frac{\max f''}{\min f''}}^2} \\
\intertext{In the last step we used the upper bound on $\pd{a}{u}$ and that
$\frac{\max f''}{\min f''}\ge 1$. From \eqref{eq:app_bound_below} we conclude that}
\tilde u - u_2&\ge \frac 12\frac{(x_4-x_2)(u_3-u_2)}{x_4-x_1}
\prn{\frac{\min f''}{\max f''}}^4.
\end{align*}
Similarly, one can show that $u_3-\tilde u \ge
\frac12\frac{(x_3-x_1)(u_3-u_2)}{x_4-x_1}\prn{\frac{\min f''}{\max f''}}^4$. 
\end{proof}
\begin{proof}[of Lemma~\ref{lem:merge:upper_bound}]
Again, WLOG we assume that $u_3\ge u_2$. 
This time we first bound $\abs{C(\tilde u)-A}$ from above:
\begin{align*}
\abs{C(\tilde u)-A}&=C(\tilde u)-B(\tilde u) \\
&=\frac{x_3-x_2}{2}(a(u_1,\tilde u)+a(\tilde u, u_4)-2a(\tilde u,\tilde u))\\
&\le (x_3-x_2)\brk{\max(u_i)-\min(u_2,u_3)}.
\end{align*}
Recall that $C(u_{23})=A$, and that $C'>0$ due to the monotonicity of $a$. 
Thus, for some point $\xi$
\begin{align*}
\abs{\tilde u-u_{23}}&= \frac{\abs{C(\tilde u)-C(u_{23})}}{C'(\xi)}\\
&\le \frac{(x_3-x_2)\brk{\max(u_i)-\min(u_2,u_3)}}{\min C'}\\
&\le 2\frac{(x_3-x_2)\brk{\max(u_i)-\min(u_2,u_3)}}{(x_4-x_1)}
\prn{\frac{\max f''}{\min f''}}^2.
\end{align*}
\end{proof}

%=============================================================================================
\end{document}